\documentclass[oneside]{amsart}
\usepackage{enumerate,amssymb}
\usepackage{mathrsfs}
\usepackage[scr=esstix,cal=boondox]{mathalfa}
\usepackage{mathtools}
\newtheorem{thm}{Theorem}[section]
\newtheorem{coro}[thm]{Corollary}
\newtheorem{prop}[thm]{Proposition}

\newtheorem{lem}[thm]{Lemma}
\theoremstyle{definition}
\newtheorem{defn}[thm]{Definition}
\newtheorem{ex}[thm]{Example}
\newtheorem{exs}[thm]{Examples}
\newtheorem{question}[thm]{Question}
\newtheorem*{qestion}{Question}
\newtheorem{questions}[thm]{Questions}

\newcommand{\Rset}{\mathbb{R}}
\newcommand{\Nset}{\omega}
\newcommand{\Pset}{\omega^\omega}
\newcommand{\Cset}{2^\omega}
\newcommand{\nset}{\Nset_+}
\newcommand{\abs}[1]{\lvert#1\rvert}
\newcommand{\Abs}[1]{\lVert#1\rVert}
\newcommand{\seq}[1]{\langle#1\rangle}
\newcommand{\hm}{\mathscr H}
\newcommand{\eps}{\varepsilon}
\newcommand{\del}{\delta}
\newcommand{\subs}{\subseteq}
\renewcommand{\leq}{\leqslant}
\renewcommand{\geq}{\geqslant}

\DeclareMathOperator{\hdim}{\dim_{\mathsf{H}}}
\DeclareMathOperator{\diam}{diam}
\DeclareMathOperator{\HD}{\mathsf{HD}}
\DeclareMathOperator{\hd}{\mathsf{hd}}
\DeclareMathOperator{\hausm}{d_\mathsf{H}}

\newenvironment{enum}{\begin{enumerate}[\rm(i)]}{\end{enumerate}}
\newenvironment{itemyze}%
  {\begin{list}{\textbullet}{\labelwidth1ex\setlength{\leftmargin}{2.3em}}}%
  {\end{list}}
\newcommand{\si}{$\sigma$\nobreakdash-}
\newcommand{\VV}{\mathcal V}
\newcommand{\MM}{\mathcal{M}}
\newcommand{\DD}{\mathcal{D}}
\newcommand{\DDD}{\mathcal{D}^{\mathsf{H}}}
\newcommand{\DDs}{\overline{\mathcal{D}}}
\newcommand{\NN}{\mathcal{N}}
\newcommand{\NNs}{\mathcal{N}_\sigma}
\newcommand{\suma}{\mathcal{I}_{1\mkern-3mu/\mkern-3mu n}}
\newcommand{\clos}[1]{\overline{#1}}
\newcommand{\closs}{\clos s}
\newcommand{\el}[1]{\ell^{#1}}
\newcommand{\mult}[1]{{\mkern-2mu\times\mkern-2mu #1}}
\newcommand{\shift}[1]{{\mkern-2mu+\mkern-2mu #1}}

\newcommand{\wh}{\widehat}
\newcommand{\whs}{\widehat s}
\newcommand{\SEQ}{\mathbb{S}}
\newcommand{\Cube}{\mathbb{C}}
\newcommand{\cube}{\mathsf{C}}
\newcommand{\emany}{\exists^\infty}
\newcommand{\fmany}{\forall^\infty}
\newcommand{\rest}{{\restriction}}
\newcommand{\concat}{^{\mkern-1mu\smallfrown}\mkern-3mu}

\newcommand{\CS}{\mathscr{C\mkern-2mu S}}
\newcommand{\mc}{\mathcal}

\newcommand{\prece}{\preccurlyeq}

\newcommand{\harm}{\mathfrak{h}}
\newcommand{\geom}{\mathfrak{g}}
\newcommand{\micr}{\boldsymbol{\varsigma}}

\DeclareMathOperator{\micro}{\boldsymbol{\mathsf{micro}}}

\newcommand{\para}{\equiv}
\newcommand{\parti}{\mathbb{P}}
\newcommand{\lom}{\mkern-2mu/\mkern-2mu}

\begin{document}
\title
[Dominated sets, microscopic sets and Hausdorff measures]
{Dominated sets, microscopic sets and Hausdorff measures}
\author{Ond\v rej Zindulka, Piotr Nowakowski}
\address
{Ond\v rej Zindulka\\
Department of Mathematics\\
Faculty of Civil Engineering\\
Czech Technical University\\
Th\'akurova 7\\
160 00 Prague 6\\
Czech Republic\\
ORCID: 0000-0001-9002-3676}
\email{ondrej.zindulka@cvut.cz}
\urladdr{http://mat.fsv.cvut.cz/zindulka}
\address{Piotr Nowakowski\\
Faculty of Mathematics and Computer Science, University of Lodz, ul. Banacha 22, 90-238 \L \'{o}d\'{z}, Poland\\
ORCID: 0000-0002-3655-4991}
\email{piotr.nowakowski@wmii.uni.lodz.pl}

\thanks{%
}

\begin{abstract}
Let $S$ be a family of sequences of positive numbers that decrease to 0, let $X$ be a metric space and $A \subset X$. $A$ is said to be $S$-dominated if, for every $s\in S$, a countable cover $\{E_n\}$ of $E$ can be found such that $\diam E_n < s_n$ for all $n$. We examine the family of all $S$-dominated sets, denoted by $\DD(S)$. In particular, we examine the connections between $\DD(S)$ and families of sets with zero Hausdorff measure for some gauges. 
\end{abstract}

\maketitle

\section{Introduction}\label{sec:intro}

A set $M\subs\Rset$ is termed \emph{microscopic} if for every $\eps>0$ there is
a cover $\seq{E_1,E_2,E_3,\dots}$ of $E$ such that $\diam E_n<\eps^n$ for all $n$.
This definition comes from Appell's papers~\cite{MR1912017,MR2152488}
and papers~\cite{MR1909968,MR2290215} of Appell, D'Aniello and V\"ath, all except the latter dated 2001.
These papers that initiated the study of microscopic sets
have since been followed by at least 22 papers on the topic.
Basic properties of microscopic sets were established in~\cite{MR1909968,MR3204595}
and the notion was extended to Euclidean spaces in~\cite{MR3259051}.
It was shown, e.g., that
\begin{itemyze}
\item every strong measure zero set is microscopic,
\item every microscopic set has Hausdorff dimension zero,
\item and both inclusions are proper,
\item the family of all microscopic sets on $\Rset$ is a $G_\del$-based \si ideal
\item that is orthogonal to the ideal of meager sets, i.e., there is a microscopic comeager set $M\subs\Rset$,
\item sets contained in a microscopic $F_\sigma$-set also form a \si ideal.
\end{itemyze}
Papers~\cite{MR3823475,MR2543905} present a number of equivalent definitions of microscopic set. 
The duality of meager and microscopic sets has been studied in
\cite{MR3204595,MR2904079,MR2882548}, Fubini properties 
in~\cite{MR3529313,MR3869249} and typical properties in~\cite{MR2429522,MR4122477}. Microscopic Cantor sets are subject of study in~\cite{MR4036593,MR4324482}.

The first application of microscopic sets~\cite{MR2058531} appeared  
as soon as in 2003. Other applications occurred in~\cite{pospisil,MR4426649,MR3384702}.

Attempts to generalize the notion culminated in the Horbaczewska's paper~\cite{MR3759529}. She merged her ideas~\cite{MR3242549} with ideas of Karasi\'nska, Paszkiewicz and Wagner Bojakowska~\cite{MR3685162}.

Deep structural properties of microscopic sets and their generalizations were described in~\cite{MR3568089} and in a very interesting Kwela's paper~\cite{MR3482702} in which the additivity of the ideal of microscopic sets was calculated. Other cardinal invariants of the ideal were estimated in~\cite{MR3823475}.

Ever since 2001 it has been obvious that the definition of microscopic sets is somehow similar but not quite the same as that of Hausdorff measure. The natural question thus has been:
\begin{qestion}
If $\phi$ is a gauge, denote by $\hm^\phi$ the Hausdorff measure on the
line induced by $\phi$ (see Section~\ref{sec:haus}). Is there $\phi$
such that a set $X\subs\Rset$ is microscopic if and only if $\hm^\phi(X)=0$?
\end{qestion}
We single out this question because it is behind most of the outcomes of the present paper.
It seems that the question has not been addressed yet. However, there has been a strong indication that the answer is negative: Kwela~\cite{MR3482702} found that the additivity of the ideal of microscopic sets is $\omega_1$, while it is a well-known fact that the additivity of (any) Hausdorff measure is consistently larger. So the two ideals are consistently different.
Moreover, Jaroslaw Swaczyna (unpublished) found that Woodin absoluteness applies, so, without going into details, Hausdorff null ideals are never equal to that of microscopic sets, modulo the assumptions of Woodin Absolutness Theorem, i.e., the existence of many specific large cardinals.
Shortly after Swaczyna came with his result (actually a few hours), 
the first ZFC proof was discovered. 

\medskip
We now provide a brief definition of the dominated sets and summarize the paper.
Our definition follows, more or less, Horbaczewska's paper~\cite{MR3759529}, 
though we adjust some details, terminology and notation.

Let $\SEQ$ be the family of all sequences of positive numbers strictly decreasing
to zero and $S\subs\SEQ$. A set $M\subs X$ in a metric space $X$ is termed 
\emph{$S$-dominated} if for each $s\in S$ there is a countable cover 
$\{E_n\}$ of $E$ such that $\diam E_n<s_n$ for all $n$. The family of all subsets of a fixed metric space is denoted by $\DD(S)$.
It is clear that if $S$ is the family of all geometric sequences, then $\DD(S)$ are microscopic sets, and also that $\DD(\SEQ)$ are strong measure zero sets. 

The goal of the paper is to systematically study the dominated sets and their relation to Hausdorff measures and, in particular, to answer the above question.

In Section~\ref{sec:dominated} we provide a detailed definition of $S$-dominated sets and underlying notions, provide some examples and 
establish some elementary facts. It is almost immediate that $\DD(S)$ does not have to be an ideal even for $S$ consisting of a single sequence. However, as proved in this section, if $S$ is closed under multiplicative shifts, then $\DD(S)$ is a \si ideal. (A multiplicative $k$-shift of a sequence $\seq{s_n}$ is the sequence $\seq{s_{kn}}$.)

Sections~\ref{sec:haus} and~\ref{sec:gauges} recall Hausdorff measures and prepare some facts regarding Cantor cubes, Cantor sets and gauges.

Let $\NN(\hm^\phi)$ the ideal of sets satisfying $\hm^\phi(X)=0$. In Section~\ref{sec:versus} it is shown that for any sequence $s\in\SEQ$
inclusion $\DD(\closs)\subs\NN(\hm^\phi)$ ($\closs$ being the set of all multiplicative shifts of $s$) is decided by the summability of the sequence $\phi\circ s$.
We also deduce that the answer to the above question is negative and
that $\NN(\hm^\phi)$ is fully determined by sets dominated by sequences
$s$ for which $\phi\circ s$ is summable, in more detail, for any gauge $\phi$
\[
  \NN(\hm^\phi)=\bigcup_{\phi\circ s\in\el1}\DD(\closs).
\]

We already know that Hausdorff measure zero and dominated sets are never the same notion. So it may be a bit of surprise that a little tweak to the definition of dominated set actually yields Hausdorff measure zero for an appropriate gauge. This is done in Section~\ref{sec:sharp}.

Section~\ref{sec:shift} deals with the families of sets contained in 
a \si compact dominated sets and with $S$-dominated sets induced by families $S$ that contain, along each $s\in S$, also all additive shifts of $s$. (An additive $k$-shift of a sequence $\seq{s_n}$ is the sequence $\seq{s_{n+k}}$.)

Section~\ref{sec:dim} describes how to determine Hausdorff dimension 
in terms of dominated sets. A conjecture attributed to Paszkiewicz
and an implicit problem of Kwela~\cite{MR3482702} are resolved.

In Section~\ref{sec:rem} we provide some comments and list a few questions that we consider interesting.
\begin{itemyze}
\item We prove that if $S$ is countably determined, then being $S$-dominated is a typical property in the Hausdorff-Gromov space, extending and generalizing thus the main result of~\cite{MR4122477},
\item comment on theorem of~\cite{MR2429522} regarding microscopic sets and typical continuous functions,
\item resolve a question of~\cite{MR4036593} regarding microscopic Cantor sets.
\end{itemyze}
We also single out several open problems regarding 
\begin{itemyze}
\item relation of dominated sets and Hausdorff measures,
\item additively complete sets,
\end{itemyze}
and maybe the most interesting question about microscopic sets:
Does a non-microscopic Borel set contain a non-microscopic compact set?

\section{Dominated sets}\label{sec:dominated}
In this section we define dominated sets and establish their elementary properties.

Throughout the paper we are using the standard notation. Specifically, $\Nset$ denotes the set of all natural numbers including zero, while $\nset$ denotes the set of all natural numbers excluding zero; $\abs{A}$ denotes the cardinality of a (usually finite) set $A$.
If $X$ is a metric space and $E\subs X$, then $\diam E$ is the diameter of $E$. 
We will make use of the $\el{p}$-spaces (also for $p<1$) over $\nset$
and in particular $\el1$ and $\el2$.
We are using $\log x$ to denote the \emph{binary logarithm} of $x$.

We begin with some notions and notation regarding sequences.
\begin{defn}\label{def:dom}
Let $\SEQ\subs(0,\infty)^{\nset}$ denote the set of all strictly decreasing sequences of positive reals converging to $0$.
We impose the following order on $\SEQ$: $s\leq s'$ if $\forall n\ s_n\leq s'_n$;
$s<s'$ is defined likewise. Clearly $\SEQ$ is closed under multiplication by a positive scalar and a positive power. These operations are denoted by
$qs$ and $s^q$. We also write $qS=\{qs:q\in S\}$.
\begin{itemyze}
\item Write $S'\leq S$ if $\forall s\in S\ \exists s'\in S'\ s'\leq s$ and
$S'\para S$ if $S'\leq S\leq S'$. Say that $S$ is \emph{countably determined}
if there is a countable $S'\para S$.
\item For $s\in\SEQ$ and $k\in\nset$ let $s^{\mult k}=\seq{s_{kn}:n\in\nset}$ be the \emph{multiplicative $k$-shift} of $s$.
Write $s^{\mult{}}$ instead of $s^{\mult2}$.
\item A set $S\subs\SEQ$ is \emph{multiplicatively complete} if
$\forall s\in S\ \exists s'\in S\ s'\leq s^{\mult{}}$. Note that if $S$ is multiplicatively complete, then
$\forall s\in S\ \forall k\ \exists s'\in S\ s'\leq s^{\mult k}$.
\item For $S\subs\SEQ$ we define the \emph{closure of $S$} to be
$\clos S=\{s^{\mult k}:s\in S, k\in\nset\}$. It is clear that $\clos S$ is multiplicatively complete and
$S\subs\clos S$. (However, $\clos S$ is not the smallest multiplicatively complete set containing $S$
-- there is no such set.)
For $s\in\SEQ$ we write $\closs$ instead of $\clos{\{s\}}$.
Note that $\closs=\{s^{\mult k}:k\in\nset\}$ and that $\closs$ is countably determined.
\end{itemyze}
\end{defn}
Besides $\SEQ$ we will also make use of the $\el{p}$-spaces (also for $p<1$)
and in particular $\el1$ and $\el2$.

Horbaczewska~\cite{MR3759529} defines $S$-dominated sets only for $S\subs\SEQ\cap\el1$
and only in the line $\Rset$.
The following definition is a direct extension of her definition to arbitrary sets
$S\subs\SEQ$ and arbitrary separable metric spaces.
\begin{defn}
Let $X$ be a separable metric space.
\begin{itemyze}
\item Let $s\in\SEQ$. A sequence $\seq{E_n:n\in\nset}$ of sets in $X$ is \emph{$s$-fine} if
$\diam E_n<s_n$ for all $n\in\nset$.
\item Let $S\subs\SEQ$. A set $E\subs X$ is \emph{$S$-dominated} if for each $s\in S$ there is
an $s$-fine cover of $E$.
\item The family of all $S$-dominated sets in $X$ is denoted by $\DD_X(S)$ or,
if there is no danger of confusion, just $\DD(S)$.
\item The family of all sets in $X$ that are contained in a \si compact $S$-dominated set
is denoted by $\DDs_X(S)$ or just $\DDs(S)$. Abusing slightly language, we will call
elements of $\DDs(S)$ \emph{\si compact $S$-dominated} sets.
Considering $\DDs(S)$ makes good sense namely in Polish spaces.
Clearly $\DDs(S)\subs\DD(S)$ and except trivial cases the inclusion is proper.
\end{itemyze}
\end{defn}
A few elementary remarks are in place:
it is clear that the witnessing $s$-fine sequences may be supposed to consist of open sets,
as well as closed sets.
It is also obvious that if $S$ is
multiplicatively complete, then one may replace ``$\diam E_n<s_n$'' with ``$\diam E_n\leq s_n$''
in the definition of $s$-fine.
Note that if $S'\leq S$, then $\DD(S')\subs\DD(S)$. Therefore, if $S'\para S$, then
$\DD(S')=\DD(S)$.

We provide a few examples of $S$-dominated sets.
Let us set a notation for two important sequences:
\begin{itemyze}
\item $\geom=\seq{2^{-n}:n\in\nset}$ denotes the binary geometric sequence,
\item $\harm=\seq{\frac1{n+1}:n\in\nset}$ denotes the harmonic sequence.
\end{itemyze}
Note that both start with $\frac12$.
\begin{exs}\label{exs1}
(i) \emph{Microscopic sets.}
Let $S$ consist of all geometric sequences.
By the definition, $E$ is \emph{microscopic} if and only if it is $S$-dominated.
We have $S\para\clos\geom$ and in particular $S$ is countably determined.
Hence a set is microscopic if and only if it is $\clos\geom$-dominated. The family $\DD_X(\clos\geom)$ of all microscopic subsets of $X$
will be occasionally denoted $\micro(X)$ or just $\micro$. Microscopic sets serve as inspiration and template for dominated sets,
but otherwise there is not much particularly special about them.

(ii) \emph{Hausdorff dimension zero.} This example comes from~\cite[4.1]{MR3482702}. Let $S$ consist of all sequences of the form $\seq{\eps^{\ln(n+1)}:n\in\nset}$ where
$\eps\in(0,1)$. The set $S$ is multiplicatively complete and countably determined. The $S$-dominated sets are
denoted by $\MM_{\ln}$ in~\cite[4.1]{MR3482702}. Routine check shows that
$\MM_{\ln}=\DD(\{\harm^q:q>0\})$. We shall see later (Proposition~\ref{dim1})
that $\MM_{\ln}$ is actually the family of all sets of Hausdorff dimension zero.

(iii) \emph{Strong measure zero.}
$\DD(\SEQ)$ is the family of all strong measure zero sets, cf.~\cite{MR4146582}.
By the Perfect Set Theorem, $\DDs(\SEQ)$ is the family of all countable sets.

(iv) \emph{Nanoscopic sets.}
Let $S$ consist of all sequences of the form $\seq{\eps^{2^n}:n\in\nset}$ where
$\eps\in(0,1)$.
It is countably determined, but not multiplicatively complete.
The $S$-dominated sets were introduced in~\cite[3.1]{MR3568089} and called
\emph{nanoscopic} sets. By~\cite[3.6]{MR3568089} nanoscopic sets do not form an ideal,
while by~\cite[Theorem 13]{MR3242549} the \si compact nanoscopic sets do form a \si ideal.

(v) \emph{Picoscopic sets.}
Let $S$ consist of all sequences of the form $\seq{\eps^{n!}:n\in\nset}$ where
$\eps\in(0,1)$.
It is countably determined, but not multiplicatively complete.
The $S$-dominated sets were introduced in~\cite[4.1]{MR3568089} and called
\emph{picoscopic} sets. By~\cite[4.2]{MR3568089} neither picoscopic nor \si compact picoscopic
sets form an ideal. Actually, there are compact sets that differ by a single point, one
picoscopic and the other not.
\end{exs}

The examples show that the families of $S$-dominated and \si compact $S$-dominated sets
need not form an ideal. The following is a sufficient condition.
\begin{thm} \label{thmideals}
Let $S\subs\SEQ$.
\begin{enum}
\item If $S\subs\SEQ$ is multiplicatively complete, then $\DD(S)$ and $\DDs(S)$ are \si ideals,
\item If $S$ is countably determined, then $\DD(S)$ is $G_\del$-based, i.e., for each
$S$-dominated set $E$ there is an $S$-dominated $G_\del$-set $G\supseteq E$.
\end{enum}
\end{thm}
\begin{proof}
(i) Let $\{A_k:k\in\nset\}\subs\DD(S)$ and $A=\bigcup_{k\in\nset}A_k$. Let $s\in S$.
For $k\in\Nset$ let $\seq{E_j^k:j\in\nset}$ be an $s^{\mult 2^{k+2}}$-fine cover
of $A_k$ (which exists because $S$ is multiplicatively complete).
For $n\in\nset$ there are unique $k,i$ such that $n=2^k(2i+1)$.
Let $F_n=E_i^k$. Then $\seq{F_n:n\in\nset}$ is a cover of $A$ and
\[
  \diam F_n=\diam E_i^k<s_{2^{k+2}i}\leq s_{2^k(2i+1)}=s_n,
\]
so $\seq{F_n}$ is $s$-fine.

(ii) We may suppose $S$ be countable. Let $E$ be $S$-dominated.
For every $s\in S$ let $\seq{E_n^s:n\in\nset}$ be an open $s$-fine cover of $E$.
Then $G=\bigcap_{s\in S}\bigcup_{n\in\nset}E_n^s$ is clearly $G_\del$, $S$-dominated, and
$E\subs G$.
\end{proof}
Note that (ii) yields, in any Polish space $X$, a partition of $X$ into two sets
$A,B\subs X$ such that $A$ is $S$-dominated and $B$ is meager.

Recall that a sequence $\seq{E_n:n\in\nset}$ of sets in $X$ is a \emph{large cover}
(or a \emph{$\lambda$-cover})
of $E$ if every $x\in E$ is contained in infinitely many $E_n$'s, i.e.,
$\forall x\in E\ \emany n\ x\in E_n$.
\begin{prop}
Let $S\subs\SEQ$ be multiplicatively complete. If $E\subs X$ is $S$-dominated,
then for each $s\in S$ there is a large $s$-fine cover of $E$.
\end{prop}
\begin{proof}
Let $E$ be $S$-dominated and $s\in S$. For every $k\in\Nset$ let
$\seq{E^k_i:i\in\nset}$ be an $s^{\mult{2^{k+2}}}$-fine cover of $E$.
Let $F_n=E^k_i$, where $2^k(2i+1)=n$. Since
\[
  \diam F_n=\diam E^k_i\leq s^{\mult{2^{k+2}}}_i=s_{2^{k+2}i}\leq s_{2^k(2i+1)}=s_n,
\]
the sequence $\seq{F_n:n\in\nset}$ is $s$-fine and it is clearly a large cover of $E$.
\end{proof}

\begin{defn}
\begin{itemyze}
\item A set $S\subs\SEQ$ is \emph{doubling} if
$\forall s\in S\ \exists s'\in S\ s'\leq\frac12s$.
\item A sequence $s\in\SEQ$ is \emph{doubling} if $\closs$ is doubling.
Note that $s$ is doubling if and only if $\exists k\ s^{\mult k}\leq\frac12s$.
\item A sequence $s\in\SEQ$ is \emph{$2$-doubling} if $s^\mult{}\leq \frac12s$.
\end{itemyze}
\end{defn}

\subsection*{Mappings}
It is easy to show that $S$-dominated sets are preserved by $1$-Lipschitz mappings for any
$S\subs\SEQ$.

We want to establish a simple correspondence between in $\Rset$ and $\Cset$.
This correspondence is carried by the canonical mapping $T\colon \Cset\to[0,1]$ defined by
\begin{equation}\label{T}
T(x)=\tfrac12\sum_{n\in\Nset} 2^{-n}x(n).
\end{equation}
We will need a simple nice lemma about $T$.

\begin{lem}[{\cite{invariants2}}]\label{3sets}
If $E\subs\Cset$, then there are sets $E_0,E_1,E_2\subs\Cset$
such that $\diam E_i\leq\diam E$ for all $i=0,1,2$ and $T^{-1}(E)\subs E_0\cup E_1\cup E_2$.
\end{lem}
Let $m\in\nset$ and consider the $m$-th cartesian power $T^m$ of $T$.
Then we have
\begin{lem}\label{nowik22}
Let $S\subs\SEQ$ be multiplicatively complete and $E\subs(\Cset)^m$. Then $E$ is $S$-dominated if and only if
$T^m(E)\subs[0,1]^m$ is $S$-dominated.
\end{lem}
\begin{proof}
Since $T$ is $1$-Lipschitz, so is $T^m$, which yields the forward implication.
As to the reverse implication, suppose that $T^m(E)$ is $S$-dominated.
Let $s\in S$. Since $S$ is multiplicatively complete, there is an $s^{\mult 3^m}$-fine cover
$\seq{F_n:n\in\nset}$ of $T^m(E)$. According to Lemma~\ref{3sets}
the preimage of each $F_n$ is covered by $3^m$ many sets $F_{n,i}$, $i<3^m$, each of
diameter at most $s_{3^mn}$. For each $j\in\nset$ there are unique $n\in\nset$, $i<3^m$
such that $j=3^mn+i$. Let $E_j=F_{n,i}$. Routine check proves that $\seq{E_j:j\in\nset}$
is an $s$-fine cover of $E$.
\end{proof}

\section{Hausdorff measures}\label{sec:haus}

In the next section we will start investigating how dominated sets and 
Hausdorff measures are interwined. We first need to recall Hausdorff measures and related background material.

\subsection*{Gauges}
A non-decreasing, right-continuous function $\phi\colon [0,\infty)\to[0,\infty)$
such that $\phi(0)=0$ and $\phi(r)>0$ if $r>0$ is called a \emph{gauge}.
The following is the common (partial) ordering of gauges, cf.~\cite{MR0281862}:
$$
  \phi\preccurlyeq\psi\ \overset{\mathrm{def}}{\equiv}
  \ \limsup_{r\to0+}\frac{\psi(r)}{\phi(r)}<\infty,\qquad
  \phi\prec\psi\ \overset{\mathrm{def}}{\equiv}
  \ \limsup_{r\to0+}\frac{\psi(r)}{\phi(r)}=0
$$
and write $\phi\approx\psi$ if $\phi\prece\psi$ and $\psi\prece\phi$.

In the case when $\phi(r)=r^s$ for some $s>0$,
then we identify $\phi$ with $s$ and
we write $\phi\prece s$ instead of $\phi\prece\psi$;
and likewise for $\phi\succcurlyeq s$.

Given $L>0$, a gauge $\phi$ is termed \emph{$L$-doubling} if
$\phi(2r)\leq L\phi(r)$ for all $r>0$ small enough%
\footnote{This means $\exists R>0\ \forall r\in(0,R]\ \phi(2r)\leq L\phi(r)$. This kind of quantification will be used frequently.}.
A \emph{doubling gauge} is a gauge that is $L$-doubling for some $L$.
We will also make use of a slight modification of the notion of $L$-doubling gauge:
call a gauge $\phi$ \emph{strictly $L$-doubling} if $\phi(2r)<L\phi(r)$
for all $r>0$ small enough.

We will need a number of facts regarding gauge. They are stated and proved in the Appendix.

\subsection*{Hausdorff measures}
Besides standard $s$-dimensional Hausdorff measures we will also make use of those
induced by gauges.
If $\del>0$, a cover $\mc A$ of a set $E\subs X$ is termed a
\emph{$\del$-fine cover} if $\diam A<\del$ for all $A\in\mc A$.
If $\phi$ is a gauge,
the \emph{$\phi$-dimensional Hausdorff measure} $\hm^\phi(E)$ of
a set $E\subs X$ is defined thus:
For each $\del>0$ set
$$
  \hm^\phi_\delta(E)=
  \inf\left\{\sum\nolimits_n\phi(\diam E_n):
  \text{$\{E_n\}$ is a countable $\delta$-fine cover of $E$}\right\}
$$
and put
$$
  \hm^\phi(E)=\sup_{\delta>0}\hm^\phi_\delta(E).
$$
Properties of Hausdorff measures are well-known, see, e.g., \cite{MR0281862}.
Hausdorff measures have a silent parameter: the underlying space $X$.
Sometimes (but rarely) we make it visible by writing $\hm^\phi_X$.

We will be dealing with ideals generated by Hausdorff measures. We will write
\begin{itemyze}
\item $\NN(\hm^\phi)=\{E\subs X:\hm^\phi(E)=0\}$ for the \si ideal of Hausdorff measure zero,
\item $\NNs(\hm^\phi)$ for the \si ideal of \si finite measure $\hm^\phi$.
\end{itemyze}
We will frequently utilize the following well-known facts:
\begin{itemyze}
\item $\phi\prece\psi\implies\NN(\hm^\phi)\subs\NN(\hm^\psi)$
\item $\phi\prec\psi\implies\NNs(\hm^\phi)\subs\NN(\hm^\psi)$
\end{itemyze}

\subsection*{Cantor cubes}
Recall that a metric $d$ is an \emph{ultrametric} if the triangle inequality reads
$d(x,z)\leq\max(d(x,y),d(y,z))$.
We utilize the following simple ultrametric spaces.
Let $A$ be a set (finite or countable and often $A=2$)
and consider the set of all $A$-valued sequences $A^\Nset$.
Note that for $x,y\in A^\Nset$, $x\wedge y$ is the common initial segment of $x$ and $y$.
Thus $\abs{x\wedge y}$ is the length of this initial segment,
i.e., the first $n$ for which $x(n)\neq y(n)$.
Suppose that $r\in\SEQ$.
For $x,y\in A^\Nset$ define $d(x,y)=r_{\abs{x\wedge y}}$. It is easily verified that
$d$ is an ultrametric on $A^\Nset$.
The ultrametric space $(A^\Nset,d)$ is denoted by $\Cube(A,r)$ and termed
a \emph{Cantor cube}. In case when $r_n=\lambda^{n}$ we write $\Cube(A,\lambda)$ instead
of $\Cube(A,r)$.
We single out two particular cases: the space $\Cube(\Nset,\frac12)$ (usually identified
with the irrationals) is denoted just by $\Pset$ and the Cantor set $\Cube(2,\frac12)$
is denoted by $\Cset$.

Note also that for any set $A$, $m\in\nset$ and $r\in\SEQ$ the $m$-th cartesian power
$\Cube(A,r)$ is isometric with $\Cube(A^m,r)$.

\begin{lem}\label{cubes1}
If $A$ is a finite set and $\phi$ is a gauge, then $\hm^\phi(\Cube(A,r))=0$ if and only if
$\liminf_{n\to\infty}\abs A^n\phi(r_n)=0$.
\end{lem}

\subsection*{Symmetric Cantor sets}
If the sequence $r\in\SEQ$ satisfies $r_{n+1}\leq\frac12r_n$ for each $n$,
we also define the standard \emph{symmetric Cantor set} $\cube(r)\subs\Rset$ induced by $r$:
Let $I_\emptyset$ be a closed interval of length $r_0$
and if $p\in2^n$ and $I_p$ has been defined, let $I_{p\concat0}$ and
$I_{p\concat1}$ be intervals of length $r_{n+1}$ such that the left endpoint of
$I_{p\concat0}$ coincides with left endpoint of $I_p$ and the right endpoint of
$I_{p\concat1}$ coincides with right endpoint of $I_p$. Then let
$\cube(r)=\bigcap_{n\in\Nset}\bigcup\{I_p:p\in2^n\}$.

The symmetric Cantor set for the case when $r_n=\lambda^n$ for some $\lambda\in(0,1/2]$
is denoted by $\cube(\lambda)$. Note that $\cube(\frac12)=[0,1]$.
\begin{lem}\label{cubes2}
\begin{enum}
\item The mapping $T_r\colon \Cube(2,r)\to\cube(r)$
that assigns to every $x\in\Cset$ the (unique) point
in the intersection $\bigcap_{n\in\Nset}I_{x\rest n}$ is a $1$-Lipschitz surjection.
\item
(\cite[4.11]{MR1333890}) If $\phi$ is a continuous gauge such that $\phi(r_n)=2^{-n}$,
then $\frac14\leq\hm^\phi(\cube(r))\leq1$.
\item
(\cite[Theorem 1.14]{MR867284}) Let $\lambda\leq\frac12$.
If $2\lambda^\beta=1$, then $\hm^\beta(\cube(\lambda))=1$.
\end{enum}
\end{lem}

\subsection*{Mappings}
It is well-known that $\hm^\phi$ measure zero is preserved by $1$-Lipschitz mappings.
The following parallels Lemma~\ref{nowik22}. See~\cite{invariants2} for the proof.
\begin{lem}\label{nowik33}
If $E\subs(\Cset)^m$, then
$\hm^\phi(E)=0$ if and only if $\hm^\phi(T^m(E))=0$.
\end{lem}

\subsection*{Howroyd Theorem}
We will repeatedly utilize the following deep result of How\-royd.
Recall that a metric space $X$ is \emph{doubling} if there is $N\in\Nset$ such that
every ball in $X$ can be covered by $N$ many balls of halved radii.
Note that $\Cset$ and all Euclidean spaces are doubling.
\begin{thm}[{\cite{MR1317515}}]\label{howroyd}
Suppose that $X$ is analytic, $\phi$ is a gauge and any of the following conditions is met:
$X$ is ultrametric, $X$ is doubling, $\phi$ is doubling.
Then every Borel set $E\subs X$ such that $\hm^\phi(E)>0$ contains a compact subset
$C\subs E$ such that $0<\hm^\phi(C)<\infty$.
\end{thm}

\section{Dominated sets versus Hausdorff measure zero I}\label{sec:versus}
In this section we establish some relations between dominated sets and sets of
Hausdorff measure zero.

\begin{thm}\label{versus1}
Let $\phi$ be a gauge and $s\in\SEQ$. The following are equivalent.
\begin{enum}
\item $\phi\circ s\in\el1$,
\item for every metric space $X$, $\DD_X(\closs)\subs\NN(\hm_X^\phi)$,
\item for every Cantor cube $X$, $\DDs_X(\closs)\subs\NNs(\hm_X^\phi)$,
\item $\DD_{\el2}(\closs)\subs\NN(\hm_{\el2}^\phi)$.
\end{enum}
\end{thm}
\begin{proof}
(i)$\Rightarrow$(ii)
Suppose that $E\subs X$ is $\closs$-dominated.
For $k\in\nset$ let $\{E_{n}\}$ be an $s^{\mult k}$-fine cover of $E$. Then
$\hm_{s_{k}}^{\phi}(E)\leq\sum_{n\in\nset}\phi(s_{kn})\leq\sum_{n\geq k}\phi(s_{n})$
and since $\phi\circ s\in\el1$, letting $k\to\infty$ yields $\hm^{\phi}(E)=0$.

(ii)$\Rightarrow$(iii) is trivial.

(iii)$\Rightarrow$(i)
Suppose that $\phi\circ s\notin\el1$. We shall find a Cantor cube $X$ such that
$\DDs_X(\closs)\nsubseteq\NNs(\hm^\phi)$.

Use Lemma~\ref{ash2}(ii) to get $\psi\succ\phi$ such that $\psi\circ s\notin\el1$.
Let $a=\psi\circ s$. Since $a$ is nonincreasing, $a^{\mult k}\notin\el1$ for every $k$,
hence one can cover $[0,1]$ with intervals of length $a^{\mult k}_n$.
It follows that $[0,1]\in\DD(\clos a)$ and by Lemma~\ref{nowik22} also $\Cset\in\DD(\clos a)$.
Define $\psi^*(t)=\inf\{u:\psi(u)\geq t\}$. Straightforward calculation shows that
\begin{itemyze}
\item[(a)] $\forall t>0\ \psi^*(\psi(t))\leq t$.
\item[(b)] $\forall t>0\ \psi(\psi^*(t))\geq t$,
\end{itemyze}
Let $r_n=\psi^*(2^{-n})$ and consider the Cantor cube $\Cube(2,r)$.
Then (a) yields $\Cube(2,r)\in\DD(\closs)$ and (b) together with Lemma~\ref{cubes1}
yields $\hm^\psi(\Cube(2,r))>0$.
Since $\phi\prec\psi$, we have $\Cube(2,r)\notin\NNs(\hm^\phi)$, as desired.

(ii)$\Rightarrow$(iv) is trivial.

(iv)$\Rightarrow$(iii) This trivially follows from the following facts:
every Cantor cube is an ultrametric space and, as proved in~\cite{MR695109},
every separable ultrametric space isometrically embeds into $\el2$.
\end{proof}
Note that (iv) remains valid if $\el2$ is replaced with other separable universal
metric spaces, e.g., the Urysohn space or $C([0,1])$.

\begin{thm}\label{versus2}
Let $X$ be analytic and doubling or ultrametric and let $\phi\circ s\in\el1$.
If $X$ is not $\closs$-dominated, then $\DD_X(\closs)\subsetneqq\NN(\hm_X^\phi)$.
Consequently, $\DD_{\el2}(\closs)\subsetneqq\NN(\hm_{\el2}^\phi)$.
\end{thm}
\begin{proof}
If $\hm^\phi(X)=0$, then $X$ is a witness. If $\hm^\phi(X)>0$, Lemma~\ref{ash2} yields
a gauge $\psi\prec\phi$
such that $\psi\circ s\in\el1$. By Lemma~\ref{2gauges}
$\NN(\hm^\psi)\subsetneqq\NN(\hm^\phi)$. Since Theorem~\ref{versus1} yields
$\DD(\closs)\subs\NN(\hm^\psi)$, we have $\DD(\closs)\subsetneqq\NN(\hm^\phi)$.

To prove the inclusion regarding $\el2$,
pick a sequence $r\in\SEQ$ such that $\phi(r_n)\geq 2^{-n}$ for all $n$
and consider the Cantor cube $C=\Cube(2,r)$.
Lemma~\ref{cubes1} yields $\hm^\phi(C)>0$, hence $C$ is not $\closs$-dominated
by Theorem~\ref{versus1}.
Since $C$ is ultrametric, we have $\DD_C(\closs)\subsetneqq\NN(\hm^\phi)$
and since $C$ isometrically embeds into $\el2$, we are done.
\end{proof}

\begin{coro}\label{sigma}
For every gauge $\phi$ and every $s\in\SEQ$
\begin{enum}
\item  $\DD_{\el2}(\closs)\neq\NN(\hm_{\el2}^\phi)$,
\item  $\DD_{\el2}(\closs)\neq\NNs(\hm_{\el2}^\phi)$.
\end{enum}
\end{coro}
\begin{proof}
If $\phi\circ s\in\el1$, then Theorem~\ref{versus2} yields both (i) and (ii) follow.
If $\phi\circ s\notin\el1$, then Theorem~\ref{versus1} yields a Cantor cube
such that $\DD_X(\closs)\nsubseteq\NNs(\hm^\phi)$ which is enough for both (i) and (ii).
\end{proof}

\begin{thm}\label{versus3}
Suppose $\phi$ is a $2$-doubling gauge and $s\in\SEQ$. The following are equivalent.
\begin{enum}
\item $\phi\circ s\in\el1$,
\item $\DDs_\Rset(\closs)\subs\NN(\hm^\phi)$,
\item $\DD_\Rset(\closs)\subsetneqq\NN(\hm^\phi)$.
\end{enum}
\end{thm}
\begin{proof}
(i)$\Rightarrow$(iii)
Note that since $\phi$ is $2$-doubling, Lemma~\ref{doubling2} yields
$\phi\prece1$ and therefore $\hm^\phi(\Rset)>0$.
Consequently, the proof of Theorem~\ref{versus2} yields $\DD(\closs)\subsetneqq\NN(\hm^\phi)$.

(iii)$\Rightarrow$(ii) is trivial.

(ii)$\Rightarrow$(i)
We modify the proof of Theorem~\ref{versus1}(iii)$\Rightarrow$(i):
Lemma~\ref{2gauges} lets us start with a strictly $2$-doubling increasing continuous gauge
$\phi$.
Let $r_n=\phi^{-1}(2^{-n})$.
As argued in the mentioned proof, the Cantor cube $\Cube(2,r)$ is $\closs$-dominated.
Since $\phi$ is strictly $2$-doubling, we have $r_{n+1}<\frac12 r_n$,
hence the symmetric Cantor set $\cube(r)$ is well defined.
By Lemma~\ref{cubes2}, $\cube(r)$ is a $1$-Lipschitz image of
an $\closs$\nobreakdash-dominated set $\Cube(2,r)$ and is therefore also
$\closs$\nobreakdash-dominated,
and at the same time $\hm^\phi(\cube(r))>0$.
Overall, $\cube(r)$ witnesses $\DDs_\Rset(\closs)\neq\NN(\hm^\phi)$.
\end{proof}

The following theorem is a substantial generalization of~\cite[Theorem 6]{MR3685162},
see also~\cite[Theorem 13]{MR4185783}.
\begin{thm}\label{versus4}
Let $\phi$ be a continuous gauge. If $E\subs X$ and $\hm^\phi(E)=0$,
then there is $s\in\SEQ$
such that $\phi\circ s\in\el1$ and $E$ is $\closs$-dominated. In other words,
\[
  \NN(\hm^\phi)=\bigcup_{\phi\circ s\in\el1}\DD(\closs).
\]
\end{thm}
\begin{proof}
We may assume that $\phi$ is strictly increasing, thus it is a bijection. In particular,
it has an inverse function $\phi^{-1}$.

Since $\hm^\phi(E)=0$, there is, for each $k\in\nset$, a cover $\seq{E_n^k:n\in\nset}$
such that
\begin{itemyze}
\item[(a)] $\sum_n\phi(\diam E_n^k)<2^{-k}$.
\end{itemyze}
Choose a one-to-one sequence $\seq{\del_n^k:n\in\nset}$ of positive numbers such that
\begin{itemyze}
\item[(b)] $\del_n^k>\diam E_n^k$,
\item[(c)] $\sum_n\phi(\del_n^k)<2^{-k}$,
\end{itemyze}
which is possible because of (a) and right-continuity of $\phi$.
Since $\del_n^k$ are positive and (c) yields
\begin{itemyze}
\item[(d)] $\lim_{n\to\infty}\del_n^k=0$,
\end{itemyze}
it is possible to rearrange the sequence so that it is nonincreasing, and since
it is one-to-one, the rearrangement is actually strictly decreasing.
So, we may suppose that $\seq{\del_n^k:n\in\nset}$ is a strictly decreasing sequence
satisfying (b) and (c).

For each $n\in\nset$ let $r_n=\sum_{k\in\nset}\phi(\del_{\lceil n/k\rceil}^k)$
(where $\lceil x\rceil$ is the ceiling function).
By (c), $r_1<\sum_n2^{-n}<\infty$.
Since $\seq{\del_n^k:n\in\nset}$ is strictly decreasing, the sequence
$\seq{r_n:n\in\nset}$ is also strictly decreasing.
Let $s_n=\phi^{-1}(r_n)$. Since $\phi^{-1}$ is strictly
increasing, $s_n$ is strictly decreasing.

Next we show that $\phi\circ s\in\el1$. Using definitions, Fubini Theorem
and (c) we get
\begin{align*}
\sum_{n\in\nset}\phi(s_n)&=\sum_{n\in\nset}\phi\phi^{-1}(r_n)
  =\sum_{n\in\nset}r_n
  =\sum_{n\in\nset}\sum_{k\in\nset}\phi(\del_{\lceil n/k\rceil}^k)\\
  &=\sum_{k\in\nset}\sum_{n\in\nset}\phi(\del_{\lceil n/k\rceil}^k)
  =\sum_{k\in\nset}k\sum_{m\in\nset}\phi(\del_m^k)
  <\sum_{k\in\nset}k\, 2^{-k}<\infty,
\end{align*}
which proves that $\phi\circ s\in\el1$, as desired.
This also proves that $s_n\to0$, so $s\in\SEQ$.

It remains to prove that the set $E$ is $\closs$-dominated.
Fix $k\in\nset$ and consider the cover $\seq{E_n^k:n\in\nset}$.
We have
\[
  \phi(\diam E_n^k)<\phi(\del_n^k)=\phi(\del_{\lceil kn/k\rceil}^k)
  <r_{kn},
\]
and therefore $\diam E_n^k<\phi^{-1}(r_{kn})=s_{kn}=s^{\mult k}_n$, as required.
\end{proof}
On $\Rset^m$ and $\Pset$ we can drop the continuity of $\phi$.
\begin{thm}
Let $\phi$ be a gauge. If $E\subs\Rset^m$ or $E\subs\Pset$ and $\hm^\phi(E)=0$,
then there is $s\in\SEQ$ such that $\phi\circ s\in\el1$ and $E$ is $\closs$-dominated.
In other words,
\[
  \NN(\hm^\phi)=\bigcup_{\phi\circ s\in\el1}\DD(\closs).
\]
\end{thm}
\begin{proof}
We first handle the case $E\subs\Pset$. Let $\psi$ be a continuous gauge such that
$\psi(2^{-n})=\phi(2^{-n})$ for all $n$.
Since the non-zero values of the metric on $\Pset$ are limited to the powers of $\frac12$,
the two Hausdorff measures $\hm^\psi$ and $\hm^\phi$ coincide on $\Pset$,
hence $\hm^\psi(E)=0$.
By Theorem~\ref{versus4}, there is a sequence $s\in\SEQ$ such that $E$ is $\closs$-dominated
and $\psi\circ s\in\el1$.
For each $n$ let $s_n'=\max\{2^{-i}:2^{-i}\leq s_n\}$.
Then clearly $\Abs{\phi\circ s'}_1=\Abs{\psi\circ s'}_1\leq\Abs{\psi\circ s}_1$,
so $\phi\circ s'\in\el1$.
On the other hand, if the diameter of a set is below
$s_n$, then it is at most equal to $s_n'$. Therefore, by the remark following
Definition~\ref{def:dom}, any $\closs$-dominated set
is also $\closs'$-dominated. Hence $E$ is $\closs'$-dominated.

Now assume that $E\subs\Rset^m$. We may suppose $E\subs[0,1]^m$. Since we are concerned with dominated sets and Hausdorff measure zero only,
Lemmas~\ref{nowik22} and~\ref{nowik33} let us suppose that $E\subs(\Cset)^m$.
Since $(\Cset)^m$ is isometric with $\Cube(2^m,\frac12)$ and the latter embeds isometrically into $\Pset$, we are done.
\end{proof}

\section{Dominated sets versus Hausdorff measure zero II}\label{sec:sharp}

We will see later that $\closs$-dominated sets and Hausdorff measure zero sets are
distinct ideals.
In this section, however, we prove that a little tweak to the definition of
$\closs$-dominated sets yields Hausdorff measure zero.

\begin{defn}
Let $s\in\SEQ$ and $\phi$  a gauge. Write $\phi\asymp s$ or $s\asymp\phi$ if
\[
  \exists a,b>0\ \fmany n\in\nset\quad a\leq n\phi(s_n)\leq b
\]
\end{defn}
Let us list a few elementary properties of this relation.
Suppose $s\in\SEQ$ and $\phi,\psi$ are gauges.
\begin{itemyze}
\item If $\phi(s_n)=\frac1n$, then $\phi\asymp s$.
\item If $s\asymp\phi\approx\psi$, then $s\asymp\psi$.
\item If $\phi\asymp s\asymp\psi$, then $\phi\approx\psi$
and \emph{a fortiori} $\NN(\hm^\phi)=\NN(\hm^\psi)$,
\item so if $\phi\asymp s$ and we are concerned only with null sets,
we may suppose without loss of generality that $\phi(s_n)=\frac1n$.
\item For each $s$ there is $\phi\asymp s$ that is continuous and strictly increasing.
\item For each $2$-doubling $s$ there is $\phi\asymp s$
that is strictly $2$-doubling, continuous and strictly increasing.
\end{itemyze}
\begin{prop}
Let $s\in\SEQ$ and $\phi\asymp s$.
\begin{enum}
\item If $\phi\asymp s$, then $s$ is doubling if and only if $\phi$ is doubling,
\item $s$ is $2$-doubling if and only if there is a $2$-doubling $\phi\asymp s$,
\item If $s$ is $2$-doubling then there is $\phi\asymp s$ that
is strictly $2$-doubling, continuous and strictly increasing.
\end{enum}
\end{prop}
\begin{proof}
(i) is straightforward.

(ii) Let $a_n=s_{2^n}$. Then $a_{n+1}\leq\frac12 a_n$, since $s$ is $2$-doubling.
Therefore, the gauge $\phi$ defined by
\[
  r\in[a_{n+1},a_n)\ \mapsto\ \psi(r)=2^{-n}
\]
is $2$-doubling and clearly $\phi\asymp s$. The backwards implication is straightforward.

(iii) Apply (ii) and Lemma~\ref{doubling11}.
\end{proof}
\begin{exs}\label{ex4.2}
Recall that $\log r$ denotes the binary logarithm of $r$.

(i) Let $\micr(r)=-\frac{1}{\log r}$. Then $\micr\asymp\geom$ (cf.~Example~\ref{exs1}(i)). 

(ii) Let $\phi(r)=r^\alpha$. Then $\phi\asymp\harm^{1/\alpha}$ (cf.~Example~\ref{exs1}(ii)).

(iii) Let $\phi(r)=1/\log\log\frac1r$. Then $\phi\asymp\eps^{2^n}$ for all $\eps\in(0,1)$
 (cf.~Example~\ref{exs1}(iv)).

\end{exs}
\begin{defn}
\begin{itemyze}
\item Let $s\in\SEQ$ and let $I\subs\nset$. A sequence of sets $\seq{E_n:n\in I}$ is
\emph{$s$-fine} if $\diam E_n<s_n$ for all $n\in I$.
\item A set $E$ is \emph{Hausdorff $S$-dominated} if
for each $s\in S$ and $\eps>0$ there is a set $I\subs\nset$ such that
$\sum_{n\in I}\frac1n<\eps$ and an $s$-fine cover $\seq{E_n:n\in I}$ of $E$.
\item The family of all Hausdorff $S$-dominated sets in $X$ is denoted by
$\DDD(S)$ or $\DDD_X(S)$.
\end{itemyze}
\end{defn}

If $S$ is multiplicatively complete, then there is a particularly simple equivalent definition of Hausdorff dominated sets.
Recall that $\suma=\{I\subs\nset:\sum_{n\in I} \frac1n<\infty\}$ is an ideal on the set $\nset$ that is called the\emph{ summable ideal}.
We will need the following partition of $\nset$. 
For $k\in\Nset$ define
\begin{equation}\label{partition}
  \parti_k=\{2^k(2i+1):i\in\Nset\}.
\end{equation}
So $\parti_k$ are the odd multiples of $2^k$. It is clear that $\{\parti_k:k\in\Nset\}$
is a partition of $\nset$ into infinite sets.

\begin{prop}\label{critA}
Let $S\subs\SEQ$ be multiplicatively complete. Then $E\in\DDD(S)$ if and only if
for each $s\in S$ there is $I\in\suma$ and a large $s$-fine cover $\seq{E_n:n\in I}$ of $E$.
\end{prop}
\begin{proof}
The forward implication:
For each $k\in\Nset$ let $J_k\subs\nset$ be such that $\sum_{n\in J_k}\frac1n<1$
and $\seq{D_i:i\in J_k}$ an $s^{\mult{2^{k+2}}}$-fine cover of $E$.
Let $I_k=\{2^k(2i+1):i\in\Nset\}\subs\parti_k$.
For each $n\in I_k$ let $E_n=D_i$ where $n=2^k(2i+1)$ and consider the sequence
$\seq{E_n:n\in I_k}$. It is a cover of $E$. Moreover, for each $n\in I_k$
\[
  \diam E_n=\diam E_{2^k(2i+1)}=\diam D_i<s_i^{\mult{2^{k+2}}}=s_{s^{k+2}i}
  \leq s_{2^k(2i+1)}=s_n,
\]
hence $\seq{E_n:n\in I_k}$ is $s$-fine.

Finally, let $I=\bigcup_{k\in\Nset}I_k$ and consider the sequence $\seq{E_n:n\in I}$.
By the above, it is $s$-fine and it is clearly a large cover of $E$.
Moreover,
\[
  \sum_{n\in I}\frac1n=\sum_{k\in\Nset}\sum_{n\in I_k}\frac1n=
  \sum_{k\in\Nset}2^{-k}\sum_{n\in J_k}\frac1n<\sum_{k\in\Nset}2^{-k}<\infty,
\]
so $I\in\suma$.

The reverse implication is easy: Let $s\in S$ and $\eps>0$. There is $J\in\suma$ and an
$s$-fine large cover $\seq{E_n:n\in J}$. Let $m$ be such that
$\sum_{n\in J\cap[m,\infty)}\frac1n<\eps$ and let $I=J\cap[m,\infty)$.
Since $\seq{E_n:n\in J}$ is a large cover, so is $\seq{E_n:n\in I}$.
\end{proof}

\begin{thm}\label{domhaus1}
Let $s\in\SEQ$ and $\phi$  be a gauge.
If $\phi\asymp s$, then $\DDD(\closs)=\NN(\hm^\phi)$.
\end{thm}
\begin{proof}
Suppose without loss of generality that $\phi(s_n)=\frac1n$ and $\phi$ is strictly increasing.

First, suppose that $E\in\DDD(\closs)$. There is $I\in\suma$ and an $s$-fine large cover
$\seq{E_n:n\in I}$ of $E$. Therefore,
\[
  \sum_{n\in I}\phi(\diam E_n)\leq\sum_{n\in I}\phi(s_n)\leq\sum_{n\in I}\tfrac1n<\infty
\]
which is enough for $\hm^\phi(E)=0$.

Now suppose that $\hm^\phi(E)=0$. Fix for the moment $k\in\nset$ and $\eps>0$.
There is a large cover $\{D_n:n\in\nset\}$ such that
$\sum_{n\in\nset}\phi(\diam D_n)<\frac{\eps}{2k}$. For each $n$ choose $\del_n>\diam D_n$
such that $\sum_{n\in\nset}\phi(\del_n)<\frac{\eps}{2k}$. The latter yields $\del_n\to 0$ and
since all $\del_n$'s are positive, we may rearrange them so that they form a nonincreasing
sequence. Therefore, $\phi(\del_n)<\frac{\eps}{2kn}$ for all $n$. Let
\[
  \gamma_n=\left\lfloor\frac{1}{2k\phi(\del_n)}\right\rfloor+n
\]
where $\lfloor \ \rfloor$ denotes the integer part.
The mapping $n\mapsto\gamma_n$ is clearly increasing and \emph{a fortiori} one-to-one.
Let $I=\{\gamma_n:n\in\nset\}$. For $j\in I$ define $E_j=D_n$ where $\gamma_n=j$. We have
\begin{equation}\label{ClaimA}
\sum_{j\in I}\tfrac1j=\sum_{n\in\nset}\tfrac1{\gamma_n}
  \leq\sum_{n\in\nset}2k\phi(\del_n)=2k\sum_{n\in\nset}\phi(\del_n)<2k\tfrac{\eps}{2k}
  =\eps.
\end{equation}
Our next goal is to prove that $\seq{E_j:j\in I}$ is $s^\mult{k}$-fine.
So let $j\in I$ and let $j=\gamma_n$. Since $\phi(\del_n)<\frac{\eps}{2kn}\leq\frac{1}{2kn}$,
we have $n<\frac{1}{2k\phi(\del_n)}$, hence $\gamma_n<\frac{1}{k\phi(\del_n)}$, hence
$\phi(\del_n)<\frac{1}{k\gamma_n}$. It follows that
\[
  \phi(\diam E_j)=\phi(\diam D_n)\leq\phi(\del_n)<\tfrac{1}{k\gamma_n}=\tfrac{1}{kj}
  =\phi(s_{kj}).
\]
Since $\phi$ is increasing, it follows that $\diam E_j<s_{kj}=s_j^{\mult k}$, as required.
Together with \eqref{ClaimA}
we proved that $\seq{E_j:j\in I}$ satisfies the assumptions of Lemma~\ref{critA}
and the proof is complete.
\end{proof}
So, in particular $\DDD(\closs)$ is a $G_\del$-based \si ideal.
If $S$ is multiplicatively complete, then $\DDD(S)=\bigcap_{s\in S}\DDD(\closs)$.
Thus, we have the following.
\begin{thm}\label{HNmain}
Let $S\subs\SEQ$ be multiplicatively complete.
\begin{enum}
\item $\DDD(S)$ is a \si ideal,
\item if $S$ is countably determined, then $\DDD(S)$ is $G_\del$-based.
\end{enum}
\end{thm}
\begin{coro}\label{sharp7}
Let $s\in\SEQ$, $\phi\asymp s$ and $\alpha>1$.
Then $\NN(\hm^\phi)\subs\DD(\closs)\subs\bigcap_{\alpha>1}\NN(\hm^{\phi^\alpha})$.
\end{coro}
\begin{proof}
The inclusion $\NN(\hm^\phi)\subs\DD(\closs)$
follows immediately from Theorem~\ref{domhaus1} and the trivial inclusion
$\DDD(\closs)\subs\DD(\closs)$.
Since $\phi^\alpha\circ s\in\el1$, the latter inclusion follows from Theorem~\ref{versus1}.
\end{proof}
\begin{thm}\label{sharp8}
Let $s\in\SEQ$, $\phi\asymp s$ and $\alpha>1$. Then
\begin{enum}
\item $\NN(\hm^\phi_{\el2})\subsetneqq\DD_{\el2}(\closs)\subsetneqq\bigcap_{\alpha>1}
\NN(\hm_{\el2}^{\phi^\alpha})$.
\item If $\phi$ is $L$-doubling for $L<2$, then
$\NN(\hm^\phi_{\Rset})\subsetneqq\DD_{\Rset}(\closs)\subsetneqq\bigcap_{\alpha>1}
\NN(\hm_{\Rset}^{\phi^\alpha})$.
\end{enum}
\end{thm}
\begin{proof}
(i) The inclusion $\NN(\hm^\phi)\subs\DD(\closs)$
has been just proved.
The opposite inclusion $\NN(\hm^\phi)\supseteq\DD(\closs)$ fails
by Theorem~\ref{versus1}, since $\phi\asymp s$ and thus $\phi\circ s\notin\el1$.

By Lemma~\ref{2gauges2}, there are gauges $\zeta\prec\xi$ such that
$\zeta\prec\xi\prec\phi^\alpha$ for all $\alpha>1$ and
$\zeta\circ s\in\el1$.
Therefore, Theorem~\ref{versus1} yields $\DD(\closs)\subs\NN(\hm^\zeta)$.

Clearly $\NN(\hm^\zeta)\subs\NN(\hm^\xi)$ and since $\hm^\xi(\el2)>0$,
by Lemma~\ref{2gauges}, $\NN(\hm^\zeta)\neq\NN(\hm^\xi)$.
In summary, we have
\begin{equation}\label{annoying}
\DD(\closs)\subs\NN(\hm^\zeta)\subsetneqq\NN(\hm^\xi)
  \subs\bigcap_{\alpha>1}\NN(\hm{\phi^\alpha})
\end{equation}
which yields the second inclusion in (i). The proof of (i) is complete.

(ii) The inclusion $\NN(\hm^\phi_{\Rset})\subsetneqq\DD_{\Rset}(\closs)$ follows at once
from Theorem~\ref{versus3}.
The inclusion $\DD_{\Rset}(\closs)\subs\bigcap_{\alpha>1}\NN(\hm_{\Rset}^{\phi^\alpha})$
follows from part (i).
To prove $\DD_{\Rset}(\closs)\neq\bigcap_{\alpha>1}\NN(\hm_{\Rset}^{\phi^\alpha})$
we tweak the proof of (i):
Since $\phi$ is $L$-doubling for some $L<2$, by Lemma~\ref{2gauges2}(iii),
both $\zeta$ and $\xi$ are $2$-doubling and by Lemma~\ref{doubling11}, we may assume
they are actually strictly $2$ doubling, continuous and increasing.
By Lemma~\ref{cubes2}(ii), there is a symmetric Cantor set $\cube(r)\subs\Rset$ such that
$0<\hm^\xi(\cube(r))<\infty$, therefore $\hm^\zeta(\cube(r))=0$.
It follows that $\cube(r)$ witnesses
$\DD_{\Rset}(\closs)\neq\bigcap_{\alpha>1}\NN(\hm_{\Rset}^{\phi^\alpha})$.
\end{proof}

\section{Additively complete sets}\label{sec:shift}
\begin{defn}\label{def:shift}
\begin{itemyze}
\item For $s\in\SEQ$ and $k\in\nset$ let
$s^{\shift k}=\seq{s_{n+k}:n\in\nset}$ be the \emph{additive $k$-shift} of $s$.
Write $s^+$ instead of $s^{\shift1}$.
\item A set $S\subs\SEQ$ is \emph{additively complete} if
$\forall s\in S\ \exists s'\in S\ s'\leq s^+$. Note that every multiplicatively complete set is additively complete,
but, as we shall see, not vice versa.
\item For $S\subs\SEQ$ we define the \emph{shift-completion} of $S$ to be
$\wh S=\{s^{\shift k}:s\in S,k\in\nset\}$.
It is clear that $\wh S$ is additively complete and $S\subs\wh S$. Note that $\clos S\leq\wh S$.
For $s\in\SEQ$ write $\whs$ in place of $\wh{\{s\}}$.
Note that $\whs=\{s^{\shift k}:k\in\nset\}$.
\end{itemyze}
\end{defn}

In this section we study $\DD(S)$ and $\DDs(S)$ for additively complete sets $S$
and in particular $\DD(\whs)$ and $\DDs(\whs)$.
Note that since $\clos S\leq \wh S$, we have $\DD(\clos S)\subs\DD(\wh S)$.
We start with proving that if $S$ is additively complete, then $\DDs(S)$ is a \si ideal,
while $\DD(S)$ does not have to.
\begin{thm}\label{thm:shift1}
If $S\subs\SEQ$ is additively complete, then $\DDs(S)$ is a \si ideal.
\end{thm}
\begin{proof}
Let $\{A_n:n\in\nset\}\subs\DDs(S)$ and $A=\bigcup_{n\in\nset}A_n$. Let $s\in S$.
We may sup\-pose all $A_n$ be compact. For each $n\in\Nset$ find recursively
$k_n$ and an open cover $\seq{U_i:k_{n-1}\leq i< k_n}$ of $A_n$  such that $\diam U_i<s_i$
for all $i$.
(Such a cover exists because $A_n$ is $\whs$-dominated and
it is finite because $A_n$ is compact.) The sequence $\seq{U_i:i\in\nset}$ is an $s$-fine
cover of $A$, as required.
\end{proof}
\begin{exs}\label{exs2}
(i) \emph{Nanoscopic sets} (Example~\ref{exs1}(iv)).
Let $S$ consist of all sequences of the form $\seq{\eps^{2^n}:n\in\nset}$ where
$\eps\in(0,1)$. Let $\mathfrak s=\seq{2^{-2^n}:n\in\nset}$ and note that
$S\para\wh{\mathfrak s}$.
Hence $\DDs(S)=\DDs(\wh{\mathfrak s})$ is a \si ideal.
We already know from Example~\ref{exs1}(iv) that $\DD(S)$ is not an ideal, so $S$ is not multiplicatively complete.

(ii) \emph{Picoscopic sets} (Example~\ref{exs1}(v)).
Let $S$ consist of all sequences of the form $\seq{\eps^{n!}:n\in\nset}$ where
$\eps\in(0,1)$. $S$ is neither multiplicatively complete nor additively complete and
and none of the families $\DD(S)$ and $\DDs(S)$ is an ideal.
\end{exs}

Consider now an arbitrary $s\in\SEQ$ and the two sets $\closs,\whs\subs\SEQ$ induced by $s$.
The families $\DD(\closs)$, $\DDs(\closs)$ and $\DDs(\whs)$ are \si ideals
(while by Example~\ref{exs2}(iv), $\DD(\whs)$ does not have to be an ideal).
Since $\closs\leq \whs$, we have $\DD(\closs)\subs\DD(\whs)$, but in general
we expect that $\DD(\closs)$ and $\DD(\whs)$ differ, and likewise
$\DDs(\closs)$ and $\DDs(\whs)$. The following theorem, which is a continuation of
Theorem~\ref{versus1}, shows that Hausdorff measures do not distinguish the two families.

\begin{thm}\label{versus3.5}
Let $\phi$ be a gauge and $s\in\SEQ$. The following are equivalent.
\begin{enum}
\item For every metric space $X$, $\DD_X(\whs)\subs\NN(\hm^\phi)$,
\item for every metric space $X$, $\DD_X(\closs)\subs\NN(\hm^\phi)$.
\end{enum}
\end{thm}
\begin{proof}
(i)$\Rightarrow$(ii) from the remarks preceding the theorem.

(ii)$\Rightarrow$(i) By Theorem~\ref{versus1}, it is enough to show that
if $\phi\circ s\in\el1$, then $\DD_X(\whs)\subs\NN(\hm^\phi)$.
So suppose that $E\subs X$ is $\whs$-dominated.
For $k\in\nset$ let $\{E_{n}\}$ be an $s^{\shift k}$-fine cover of $E$. Then
$\hm_{s_{k}}^{\phi}(E)\leq\sum_{n\in\nset}\phi(s_{n}^{\shift k})\leq\sum_{n\geq k}\phi(s_{n})$
and letting $k\to\infty$ yields $\hm^{\phi}(E)=0$.
\end{proof}

The following important lemma from \cite[Lemma 3.8]{MR3568089} is a basic
tool for \si compact dominated sets.
We adjust it to our setting, but omit the proof, because it is almost
identical to that of~\cite[Lemma 3.8]{MR3568089}.
\begin{lem}\label{2mic}
Let $s\in\SEQ$, $m\in\Nset$. If $E\subs X$ is a compact set such that
for~every~$k$ there are $s^{\shift k}$-fine sequences of sets
$\seq{E_n^i:n\in\nset}$, $i<m$, such that $\{E_m^i:i<m,n\in\nset\}$ is a cover of $E$,
then $E$ is a union of $m$ compact $\whs$-dominated sets.
\end{lem}
The lemma lets us prove the analogy to Lemma~\ref{nowik22} for additively complete
sets.
\begin{lem}\label{nowik44}
Let $S\subs\SEQ$ be additively complete and $E\subs(\Cset)^m$.
Then $E\in\DDs(S)$ if and only if
$T^m(E)\in\DDs(S)$.
\end{lem}
\begin{proof}
The forward implication is straightforward. Backward implication:
We may assume that $E$, and hence also $T^m(E)$, are compact.
Let $s\in S$ and let $\seq{E_n:n\in\nset}$ be an $s$-fine cover of $T^m(E)$.
Then, by Lemma~\ref{3sets}, $E$ is covered by $3^m$ many $s$-fine families.
Since $S$ is additively complete, the argument holds also for all $s^{\shift k}$.
Thus, by Lemma~\ref{2mic}, $E$ is covered by $3^m$ compact $\whs$-dominated sets
and therefore is, by Theorem~\ref{thm:shift1}, compact and $\whs$-dominated.
Hence $E\in\bigcap_{s\in S}\DDs(\whs)=\DDs(S)$.
\end{proof}
%
We now look into the following interrelated questions:
\begin{itemyze}
\item For which sequences is $\DD(\whs)$ a \si ideal?
\item For which sequences is $\DD(\whs)=\DD(\closs)$?
\item For which sequences is $\DDs(\whs)=\DDs(\closs)$?
\end{itemyze}
We have only a few partial answers. Theorems~\ref{s-nots*} and \ref{s-nots*2}
address the first and second question. Both say that if $s$ converges fast enough, then
$\DD(\whs)\neq\DD(\closs)$, while Theorem~\ref{thmwarpj} addresses the third question
saing that if $s$ converges slow enough, then $\DDs(\whs)=\DDs(\closs)$.

The proof of the following theorem is very similar to the proof of
\cite[Lemma 3.10]{MR3568089}, where $s_n=2^{-2^n}$.
%
\begin{thm} \label{s-nots*}
    Let $s\in \SEQ$ be a sequence such that 
    \begin{equation} \label{1sn>}
    \fmany n\quad s_n \geq 4^ns_{n+1}.
    \end{equation}
Then 
\begin{enum}
\item $\DD(\whs)$ is not an ideal;
\item $\DD(\clos s)\neq\DD(\whs)$.
\end{enum}
\end{thm}
\begin{proof}
Since $\DD(\whs)=\DD(\wh{s^{+n}})$, for any $n\in\nset$ we may assume that $s_n \geq 4^ns_{n+1}$ for all $n$ (if not, then we can change first few terms of the sequence). Let $T_{-1} = \{0,1\}$ and
$T_i=\{2^{i+1},2^{i+1}+1,\dots,2^{i+2}-1\}$ for $i\in\Nset$.
Let $S_0=T_{-1}$ and $S_{i+1}=\bigcup_{j\in S_i} T_j$ for $i\in\nset$.
    Observe that both sequences $(T_i)$ and $(S_i)$ form partitions of $\Nset$ and $T_n$ consists of $2^{n+1}$ numbers for $n\in\Nset$.
    Let $f \colon\Nset \to\Nset$ be given by the formula $f(k)=2^{i+1}$, where $i\in\Nset \cup \{-1\}$ is the unique number such that $k\in T_i$, so for $k>1$ $f(k)$ is the cardinality of $T_i$. In particular, for $k\in \{2^n,2^n+1,\dots,2^{n+1}-1\}$
    $f(k) = 2^n$, and so $f(2k)\geq k+1$ and $f(2k+1)\geq k+1$.
    
Let $I_{-1}= [0,3s_1]$ and $(I_n)_{n\in\Nset}$ be a sequence of closed intervals such that 
\begin{itemyze}
\item for all $n\in\Nset\cup\{-1\}$ and all $k\in T_n$, $I_k\subs I_n$, intervals $I_k$ are pairwise disjoint and the distance between them is not less than $s_{f(n)+1}$. In particular, if $i,j\in S_n$, $i \neq j$, then $I_i\cap I_j=\emptyset$.
\item $\diam I_{2k}=\diam I_{2k+1}=s_{k+1}$, for all $k\in\Nset$.
\end{itemyze}
Such a construction is possible because
for any $n\in\Nset$ we have, by (\ref{1sn>}),
\begin{align*}   
  \diam I_{2n} &=s_{n+1} \geq 
  4^{n+1}s_{n+2}=2\cdot 4^n s_{n+2}+2\cdot4^n s_{n+2} \\
  &> 2 \cdot 4^n s_{4^n+1}+ (2^{2n+1}-1)s_{n+2} \\
  &\geq 2\sum_{i=2^{2n}+1}^{2^{2n+1}} s_i +(2^{2n+1}-1)s_{n+2} 
  \geq \sum_{k\in T_{2n}} \diam I_k + (2^{2n+1}-1)s_{f(n)+1}
\end{align*}
and similarly, 
\begin{align*}
    \diam I_{2n+1} &= s_{n+1} > 2\sum_{i=2^{2n+1}+1}^{2^{2n+2}} s_i +(2^{2n+2}-1)s_{n+2} \\
    &\geq \sum_{k\in T_{2n+1}}\diam I_k +(2^{2n+2}-1)s_{f(n)+1}. 
\end{align*}
For $n\in\Nset$ let $X_n = \bigcup_{k\in S_n} I_k$ and $X = \bigcap_{n\in\Nset} X_n$.

Observe that $X$ is compact and for any $n\in\Nset$ $X$ is covered by finitely many intervals of length $s_k, s_k, s_{k+1}, s_{k+1},\dots,s_m,s_m$, where $k = \frac{\min S_n}{2}$, $m=\frac{\max S_n-1}{2}$. By Lemma \ref{2mic}, $X$ is a union of two sets from $\DD(\whs)$. 

We will now show that $X\notin\DD(\whs)$.  Let $\seq{J_n}$ be a sequence of intervals such that $\diam J_n<s_n$. Since $\diam J_1<s_1$ and the distance between $I_0$ and $I_1$ is not less than $s_1$, then $J_1$ can intersect only one of $I_0$ and $I_1$. Let $m\in\{0,1\}$ be such that $J_1\cap I_m=\emptyset$. Let $K_0=I_m$. Then $K_0\cap\bigcup_{i\leq f(m)}J_i=K_0\cap J_1=\emptyset$.

Now suppose that for $n\in\Nset$ we have defined a decreasing sequence of closed intervals $\seq{K_i:0\leq i\leq n}$ with $K_n=I_j$ for some $j\in S_n$ and $K_n\cap\bigcup_{i\leq f(j)} J_i=\emptyset$. 
$X_{n+1}\cap I_j$ is a union of $2^{j+1}$ intervals $I_i$ for $i\in T_j$ which are pairwise disjoint and the distance between them is not less than $s_{f(j)+1}$. Of course, these intervals are also disjoint with $\bigcup_{i\leq f(j)} J_i$. Since $\diam J_k<s_k\leq s_{f(j)+1}$ for $k\geq f(j)+1$ and $f(j)<2^{j+1}$, there are at least $f(j)$ intervals $I_i$ for $i\in T_j$ which are disjoint with $\bigcup_{l\leq 2^{j+1}} J_l$. Let $I_p$ be one of them. Put $K_{n+1} = I_p$. Since $f(p)=2^{j+1}$, we have $K_{n+1}\cap\bigcup_{i\leq f(p)} J_i = \emptyset$.

We have inductively constructed a sequence of intervals $\seq{K_n}$ such that $K_{n+1}\subs K_n$, all $K_n$ are intervals $I_t$ for some finite binary sequence $t$ and for some increasing sequence $\seq{k_n}$ tending to infinity we have $K_n \cap \bigcup_{i\leq k_n} J_i = \emptyset$. Consequently, 
$$
\emptyset \neq \bigcap_{n\in\Nset}K_n \subs X \setminus \bigcup_{i\in\nset} J_i,
$$
so the set $X$ is not covered by the sequence $\seq{J_i}$. Since $X$ is a union of two sets from $\DD(\whs)$, but does not belong to $\DD(\whs)$, this family is not an ideal. By Theorem \ref{thmideals}, $\DD(\clos s)$ is an ideal, so $\DD(\clos s)\neq\DD(\whs)$.
\end{proof}

%
\begin{thm} \label{s-nots*2}
Let $s\in \SEQ$ be a sequence such that
\begin{equation} \label{3sn>}
  \exists m\ \fmany n\in\Nset\quad s_{n} > 2^n s_{mn}.
\end{equation}
Then $\DD(\clos s)\neq\DD(\whs)$.
\end{thm}
\begin{proof}
Let $T_{-1} =\{0,1\}$ and $T_i=\{2^{i+1},2^{i+1}+1,\dots,2^{i+2}-1\}$ for $i\in\Nset$. Let $S_0=T_{-1}$ and $S_{i+1}=\bigcup_{j\in S_i} T_j$ for $i\in\nset$. Observe that both sequences $\seq{T_i}$ and $\seq{S_i}$ form partitions of $\Nset$ and $T_n$ consists, for $n\in\Nset$, of $2^{n+1}$ numbers. Define $f\colon\Nset\to\Nset$ by  $f(k)=2^{i+1}$, where $i\in\Nset\cup\{-1\}$ is the unique number such that $k\in T_i$, so for $k>1$, $f(k)$ is the cardinality of $T_i$. In particular, for $k\in\{2^n,2^n+1,\dots,2^{n+1}-1\}$
$f(k)=2^n$, and so $f(2k)\geq k+1$ and $f(2k+1)\geq k+1$. Therefore, $f(k)\geq\lceil\frac{k}{2}\rceil$.

Let $m, N \in \nset$ be such that $s_n > 2^ns_{mn}$ for all $n \geq N$. Let $K\geq N$ be such that for all $n\geq K$, $2^{n+1}\geq m^2n$. 
Let $I_{-1}$ be a closed interval that is long enough and $\seq{I_n:n\in\Nset}$ be a sequence of closed intervals such that 
\begin{itemyze}
\item for all $n\in\Nset\cup \{-1\}$ and all $k\in T_n$, $I_k\subs I_n$, intervals $I_k$ are pairwise disjoint and the distance between them is not less than $s_{m^2n}$ for $n>0$ and not less than $s_1$ for $n=-1$ or $n=0$. In particular, if $i,j\in S_n$, $i\neq j$, then $I_i \cap I_j=\emptyset$.
\item $\diam I_{n}= s_{n}$, for all $n\geq K$.
\end{itemyze}
Such a construction is possible because for any $n \geq K$ (for smaller $n$ we can adjust the length of $I_n$ accordingly), using  (\ref{3sn>}) twice, we get
\begin{align*}
  \diam I_{n} &= s_{n}>2^n s_{mn}>2^n\cdot2^{mn}s_{m(mn)}\geq 4\cdot 2^n s_{m^2n} \\
  &\geq 2^{n+1}s_{m^2n}+2^{n+1}s_{m^2n} \geq 2^{n+1}s_{2^{n+1}}+2^{n+1}s_{m^2n} \\
  &> \sum_{i=2^{n+1}}^{2^{n+2}-1} s_i 
        +(2^{n+1}-1)s_{m^2n} = \sum_{k\in T_{n}} \diam I_k+ (2^{n+1}-1)s_{m^2n}.
\end{align*}
For $n\in\Nset$ let $X_n=\bigcup_{k\in S_n}I_k$ and finally let $X=\bigcap_{n\in\Nset} X_n$.

Observe that $X$ is compact and for any $n\geq K$ it is covered by finitely many intervals with length equal to $s_k,s_{k+1},s_{k+2}, \dots,s_{l-1},s_l$, where $k=\min S_n$, $l=\max S_n$. Thus, $X\in \DD(\whs)$. 

We will now show that $X \notin \DD(\clos s)$.  Let $k > 2m^2$ and $\seq{J_n}$ be a sequence of intervals such that $\diam J_n< s_n$. We will show that $X\not\subs \bigcup_{i=1}^\infty J_{ki}$. Since $\diam J_k< s_k$ and the distance between $I_0$ and $I_1$ is not less than $s_1$, then $J_k$ can intersect only one of $I_0$ and $I_1$. Let $q\in \{0,1\}$ be such that $J_1 \cap I_q = \emptyset$. Let $K_0 = I_q$. Then $K_0 \cap \bigcup_{i\leq f(q)} J_{ki} = K_0 \cap J_k = \emptyset$.

Now suppose that for $n\in\Nset$ we have defined a decreasing sequence of closed intervals $\seq{K_i:0\leq i\leq n}$ with $K_n=I_j$ for some $j\in S_n$ and $K_n\cap \bigcup_{i\leq f(j)} J_{ki}=\emptyset$. 
$X_{n+1}\cap I_j$ is a union of $2^{j+1}$ intervals $I_i$ for $i\in T_j$ which are pairwise disjoint and the distance between them is not less than $s_{m^2j}$. Of course, these intervals are also disjoint with $\bigcup_{i\leq f(j)} J_{ki}$. For $n \geq f(j)+1$ we have 
\[
  \diam J_{kn}< s_{kn} \leq s_{k(f(j)+1)} \leq s_{2m^2\lceil \frac{j}{2} + 1 \rceil } < s_{m^2j}
\] 
and $f(j)<2^{j+1}$. Therefore, there are at least $f(j)$ intervals $I_i$ for $i\in T_j$ which are disjoint with $\bigcup_{n\leq 2^{j+1}} J_{kn}$. Let $I_p$ be one of them. Put $K_{n+1} = I_p$. Since $f(p)=2^{j+1}$, we have $K_{n+1} \cap \bigcup_{i\leq f(p)} J_{ki} = \emptyset$.

We have inductively constructed a sequence of intervals $\seq{K_n}$ such that $K_{n+1}\subs K_n$, all $K_n$ are intervals $I_t$ for some finite binary sequence $t$ and for some increasing sequence $\seq{k_n}$ tending to infinity we have $K_n \cap \bigcup_{i\leq k_n} J_{ki} = \emptyset$. Consequently, 
\[
  \emptyset \neq \bigcap_{n\in\Nset}K_n \subs X \setminus \bigcup_{i\in\nset} J_{ki},
\]
so the set $X$ is not covered by the sequence $\seq{J_{ki}}$. Thus, $\DD(\clos s)\neq\DD(\whs)$.
\end{proof}
\begin{thm}\label{thmwarpj}
Let $X$ be doubling and $s\in\SEQ$ be such that
\begin{equation}\label{simple3}
  \exists j\ \fmany n\in\Nset\quad s_{n+j}\leq js_{2n}.
\end{equation}
%
%
Then $\DDs(\whs)=\DDs(\closs)$.
\end{thm}
\begin{proof}
Let $E\subs X$ be compact and $\whs$-dominated and let $i\in\Nset$. Let
$\seq{E_n}$ be an $s^{\shift i}$-fine cover. Since $X$ is doubling, for any $\del>0$
there is a number $m\in\Nset$ such that each set $A\subs X$ can be covered by $m$
many sets of diameters at most $\del\diam A$, and this in particular holds for $E_n$'s.
Hence the assumptions of Lemma~\ref{2mic}
are satisfied, and therefore $E$ is a finite union of $\del \whs$-dominated sets.
Since $\del\whs=\wh{\del s}$, $E$ is, by Theorem~\ref{thm:shift1}, $\del\whs$-dominated.

Now suppose that $E\in\DDs(\whs)$. Then it is covered by countably many compact
$\whs$-dominated sets. By the above, each of them is $\del \whs$-dominated and
thus $E$ is $\del \whs$-dominated as well. It follows that $\DDs(\whs)=\DDs(\del \whs)$.

Let $\eps=1\lom j$ and $S=\bigcup_{k\in\nset}\eps^k \whs$.
We have $\DDs(S)=\bigcap_{k\in\nset}\DDs(\eps^k \whs)=\DDs(\whs)$.
An easy proof by induction shows that
\[
\forall k \ \fmany n \quad s_{2kn}\geq\eps^k s_{n+2j}.
\]
Since $\eps^k s^{\shift{2j}}\in S$, it follows that $S\leq\closs$.
Since clearly $\closs\leq \whs$,
we have
\[
  \DDs(S)\subs\DDs(\closs)\subs\DDs(\whs)=\DDs(S),
\]
hence all ideals are equal and in particular $\DDs(\closs)=\DDs(\whs)$.
\end{proof}

\begin{exs}\label{exs6}
(i) \emph{Microscopic sets.}
It is easy to show that $\geom$ satisfies condition~\eqref{3sn>}.
Thus, Theorem~\ref{s-nots*2} yields $\DD(\clos\geom)\neq\DD(\wh{\geom})$.
However, we do not know if $\DD(\wh{\geom})$ is an ideal, let alone \si ideal.

(ii) \emph{Generalized microscopic sets.}
More generally, if $p>0$ and $(\geom_p)_n=2^{-n^p}$, then condition~\eqref{3sn>} holds for $\geom_p$.
Therefore, $\DD(\clos\geom_p)\neq\DD(\wh{\geom_p})$ for all $p>0$.

(iii) \emph{Generalized microscopic sets II.}
It is easy to check that if $p\geq 2$, then condition~\eqref{1sn>} holds.
Thus, by Theorem~\ref{s-nots*}, $\DD(\wh{\geom_p})$ is not an ideal if $p \geq 2$.

(iv)  \emph{Generalized nanoscopic sets.}
Let $\eps<1$, $\beta>1$ and $s_n=\eps^{\beta^n}$, cf.~Example~\ref{exs1}(iv).
It is easy to check that~\eqref{1sn>} holds.
Thus, by Theorem~\ref{s-nots*}, $\DD(\whs)$ is not an ideal and
$\DD(\closs)\neq\DD(\whs)$.

(v) If $p>1$, then \emph{the generalized harmonic sequence} $\harm^p$ satisfies
condition~\eqref{simple3} of Theorem~\ref{thmwarpj}.
Therefore, $\DDs(\wh{\harm^p})=\DDs(\clos{\harm^p})$.
\end{exs}
In Corollary~\ref{sigma} we showed that $\DD_{\el2}(\closs)\neq\NN(\hm_{\el2}^\phi)$
for every $s$ and $\phi$. For some $s$ way more is true: there is no family $\Phi$
of gauges such that $\DD(\closs)=\bigcap_{\phi\in\Phi}\NN(\hm^\phi)$.
\begin{prop}
If $\DD(\closs)\neq\DD(\whs)$, then
\begin{equation}\label{Phi}
  \text{for every family $\Phi$ of gauges }\DD(\closs)\neq\bigcap_{\phi\in\Phi}\NN(\hm^\phi).
\end{equation}
In particular, \eqref{Phi} holds for every $s\in\SEQ$ that satisfies \eqref{3sn>}.
\end{prop}
\begin{proof}
Suppose that $\DD(\closs)=\bigcap_{\phi\in\Phi}\NN(\hm^\phi)$. Then
$\DD(\closs)\subs\NN(\hm^\phi)$ for every $\phi\in\Phi$ and, by Theorem~\ref{versus3},
$\DD(\whs)\subs\NN(\hm^\phi)$ for every $\phi\in\Phi$. Therefore,
\[
  \DD(\closs)\subs\DD(\whs)\subs\bigcap_{\phi\in\Phi}\NN(\hm^\phi)
\]
and consequently all of the sets are equal and in particular $\DD(\closs)=\DD(\whs)$:
a contradiction concluding the proof.
\end{proof}
\begin{ex}
The ideal of microscopic sets is not equal to $\bigcap_{\phi\in\Phi}\NN(\hm^\phi)$
for any family $\Phi$ of gauges.
The same holds for the generalized microscopic sets.
\end{ex}
\section{Hausdorff dimension}\label{sec:dim}
Recall that Hausdorff dimension of a set is defined and denoted by
\[
  \hdim E=\inf\{\beta:\hm^\beta(E)=0\}=\sup\{\beta:\hm^\beta(E)=\infty\}.
\]
In this section we study ideals of sets of small Hausdorff dimension on the line $\Rset$,
i.e., \si ideals
\begin{alignat*}{2}
  &\HD(\alpha)&&=\{E\subs\Rset: \hdim E\leq \alpha\},\\
  &\hd(\alpha)&&=\{E\subs\Rset: \hdim E< \alpha\}
\end{alignat*}
where $\alpha\geq0$, and characterize them in terms of dominated sets.
Recall that
\begin{itemyze}
\item $\harm=\seq{\frac1{n+1}:n\in\nset}$ is the harmonic sequence,
\item $\MM_{\ln}=\DD(\{\seq{\eps^{\ln(n+1)}:n\in\nset}:\eps>0\})$.
\end{itemyze}
\begin{prop}\label{dim1}
For each $\alpha\geq0$ let $S_\alpha=\{\harm^{1\lom\beta}:\beta>\alpha\}$.
Then
\begin{enum}
\item $\HD(\alpha)=\DD(\clos S_\alpha)$,
\item $\HD(0)=\DD(S_0)=\MM_{\ln}$.
\end{enum}
\end{prop}
\begin{proof}
(i) Let $\gamma>\beta>\alpha\geq0$.
Consider the generalized harmonic sequences $\harm^{1\lom\beta}$
and the gauges $\phi_\beta(r)=r^\beta$. Clearly
\begin{enum}
\item[(a)] $\phi_\beta\asymp\harm^{1\lom\beta}$,
\item[(b)] $\phi_\gamma\circ\harm^{1\lom\beta}=\harm^{\gamma\lom\beta}\in\el1$,
\end{enum}
hence for each $\gamma>\beta$
\[
\NN(\hm^\beta)=\DDD(\clos{\harm^{1\lom\beta}})\subs\DD(\clos{\harm^{1\lom\beta}})
  \subs\NN(\hm^\gamma)
\]
the last inclusion by (b) and Theorem~\ref{versus1}.
Therefore,
\[
  \HD(\alpha)=\bigcap_{\beta>\alpha}\NN(\hm^\beta)
  =\bigcap_{\beta>\alpha}\DD(\clos{\harm^{1\lom\beta}})=\DD(\clos S_\alpha).
\]
(ii) A routine check proves that $S_0$ is multiplicatively complete.
Hence the first equality is a particular case of (i). The second inequality follows from
$\eps^{\ln(n+1)}=(n+1)^{\ln\eps}$ and the defi\-ni\-tion of $\MM_{\ln}$.
\end{proof}
\begin{coro}
$\hdim E=\inf\{\beta:E\text{ is $\clos{\harm^{1\lom\beta}}$-dominated}\}$
for every $E\subs X$.
\end{coro}

\subsection*{Paszkiewicz Conjecture}
Horbaczewska~\cite{MR3759529} considers the following conjecture (PC1)
that she attributes to Paszkiewicz. We add one more conjecture (PC2) in the same vein.
Let $X=\Rset$ and $0<\alpha\leq1$.
\begin{itemyze}
\item ($\dagger$) $\hdim E\leq\alpha\implies\exists s\in\SEQ\ E\in\DD(\closs)\subs\HD(\alpha)$.
\item (PC1) Condition ($\dagger$) fails.
\item (PC2) The following fails: $\exists s\in\SEQ\ \HD(\alpha)=\DD(\closs)$.
\end{itemyze}
We now show that (PC1) fails, but (PC2) holds. We use the following notation:
\[
  \el{\alpha+}=\bigcap_{\beta>\alpha}\el{\beta},\quad
  \el{\alpha-}=\bigcup_{\beta<\alpha}\el{\beta}
\]

\begin{thm}\label{maindim}
Let $0\leq\alpha\leq1$ and $X=\Rset$.
\begin{enum}
\item $\HD(\alpha)=\bigcup\{\DD(\closs):s\in\SEQ\cap\el{\alpha+}\}$.
In particular, 
\emph{(PC1)} fails for all $\alpha\in[0,1]$.
\item $\hd(\alpha)=\bigcup\{\DD(\closs):s\in\SEQ\cap\el{\alpha-}\}$.
\end{enum}
\end{thm}
\begin{proof}
Let $\hdim E\leq\alpha$.
Let $\phi$ be a continuous gauge $\phi$ such that $\hm^\phi(E)=0$ and
$\phi\prec\beta$ for all $\beta>\alpha$ guaranteed by Lemma~\ref{dim2}.
Since $\phi$ is continuous, there is $s\in\SEQ$ such that $s\asymp\phi$.
Therefore, $E\in\NN(\hm^\phi)=\DDD(\closs)\subs\DD(\closs)$.

On the other hand, since $s\asymp\phi\prec r^{\alpha+\frac1m}$ for all $m\in\nset$, we have
$\forall n\ s_n\leq n^{-{1/(\alpha+\frac1m)}}$ and consequently
$\DD(\closs)\subs\DD\Bigl(\clos{\harm^{1/(\alpha+\frac1m)}}\Bigr)$.
It follows that
\[
  \DD(\closs)\subs
  \bigcap_{m\in\nset}\DD\Bigl(\clos{\harm^{1/(\alpha+\frac1m)}}\Bigr)
  =\HD(\alpha)
\]
by Proposition~\ref{dim1}. Since
$\forall\beta>\alpha\ \exists m\ \frac{\beta}{\alpha+\frac1m}>1$,
we have $s\in\el{\beta}$, so $s\in\el{\alpha+}$.

In summary,
\begin{equation}\label{aux1}
  \hdim E\leq\alpha\implies\exists s\in\el{\alpha+}\
E\in\DD(\closs)\subs\HD(\alpha).
\end{equation}

On the other hand, if $s\in\el{\beta}$, then $r^\beta\circ s\in\el1$, therefore
Theorem~\ref{versus1} yields $\DD(\closs)\subs\NN(\hm^\beta)$.
Since $\HD(\alpha)=\bigcap_{\beta>\alpha}\NN(\hm^\beta)$, it follows that
\[
  \HD(\alpha)\subs\bigcup\{\DD(\closs):s\in\el{\alpha+}\}
\]
and~\eqref{aux1} yields the reverse inclusion.

We proved (i). (ii) follows from (i) by straightforward calculation.
\end{proof}

We now prove (PS2).
\begin{thm}
\begin{enum}
\item For every $s\in\SEQ$ and each $\alpha<1$, $\HD(\alpha)\neq\DD(\closs)$.
\item For every $s\in\SEQ$, $\HD(1)\neq\DD(\closs)$ if and only if $s\in\el1$.
\end{enum}
\end{thm}
\begin{proof}
(i) Aiming towards contradiction suppose that there is $s\in\SEQ$ such that
$\HD(\alpha)=\DD(\closs)$. Let $\phi$ be a gauge such that $\phi\asymp s$.
Corollary~\ref{sharp7} yields $\NN(\hm^\phi)\subs\HD(\alpha)$.

Let $\alpha<\beta<\gamma$ and $\eps=\frac\gamma\beta-1$.
Then $\HD(\alpha)\subs\NN(\hm^\beta)$ and hence $\NN(\hm^\phi)\subs\NN(\hm^\beta)$
and Lemma~\ref{ideals} yields $\phi\prece\beta$.
Therefore, $\phi^{1+\eps}\prece\beta(1+\eps)=\gamma$.
We proved
\begin{equation}\label{pas1}
  \forall\gamma>\alpha\ \exists\eps>0\quad \phi^{1+\eps}\prece\gamma.
\end{equation}
By Lemma~\ref{2gauges2}, there are gauges $\zeta\prec\xi$ such that for any $\gamma>\alpha$
$\zeta\prec\xi\prec\phi^{1+\eps}\prece\gamma$ for any $\gamma>\alpha$ and
$\zeta\circ s\in\el1$, whence $\DD(\closs)\subs\NN(\hm^\zeta)$.
By Lemma~\ref{2gauges}, $\NN(\hm^\zeta)\subsetneqq\NN(\hm^\xi)$. In summary
\[
  \HD(\alpha)\subs\DD(\closs)\subs\NN(\hm^\zeta)\subsetneqq\NN(\hm^\xi)\subs\NN(\hm^\gamma)
\]
holds for all $\gamma>0$, so taking intersection on the right yields
\[
  \HD(\alpha)\subs\DD(\closs)\subs\NN(\hm^\zeta)\subsetneqq\NN(\hm^\xi)
  \subs\bigcap_{\gamma>\alpha}\NN(\hm^\gamma)=\HD(\alpha),
\]
a contradiction concluding the proof of (i).

(ii) If $s\in\el1$, then, by Theorem~\ref{versus1}, $\DD(\closs)\subs\NN(\hm^1)$.
The line $\Rset$ witnesses $\NN(\hm^1)\subsetneqq\HD(1)$.
If $s\notin\el1$, then $[0,1]$ is clearly $s$-dominated,
which is enough for $\DD(\closs)=\HD(1)$.
\end{proof}

\section{Appendix: Gauges}\label{sec:gauges}

In this section we establish various technical lemmas regarding gauges.
\begin{lem}\label{ash}
Let $c$ be a nonincreasing sequence of positive reals.
\begin{enum}
\item If $c\in\el1$, then there is $c'\in\SEQ\cap\el1$ such that $\frac{c_n}{c'_n}\to 0$.
\item If $c\notin\el1$, then there is $c'\in\SEQ\setminus\el1$ such that
$\frac{c_n}{c'_n}\to\infty$.
\end{enum}
\end{lem}
\begin{proof}
(i)  Let $r_n=\sum_{i\geq n}c_i$ and $d_n=\frac{c_n}{\sqrt{r_n}}$.
Then by \cite{Ash}, $d\in\el1$.
However, $d$ need not be decreasing. To fix this, let
\[
  \wh n=\max\{j\leq n: d_j=\min_{i\leq n}d_i\}.
\]
The following are immediate from the definition.
\begin{enum}
\item[(a)] $\wh n\leq n$,
\item[(b)] $d_{\wh n}\leq d_{n}$,
\item[(c)] $d_{\wh n}\geq d_{\wh{n+1}}$
\end{enum}
It is easy to show that for each $n$, either $\wh{n+1}=\wh n$ or $\wh{n+1}=n+1$ and therefore
$\seq{\wh n:n\in\nset}$ is a non-decreasing sequence that has infinitely many fixed points
and thus
\begin{enum}
\item[(d)] $\lim_{n\to\infty}\wh n=\infty$.
\end{enum}


Define $c'_{n}=d_{\wh n}+2^{-n}$. Since by (b) $c'_n\leq d_n+2^{-n}$, we have
$\Abs{c'}_1\leq\Abs{d}_1+1$ and thus $c'\in\el1$. By (c), $c'$ is strictly decreasing.
Also
$$
  \lim_{n\to\infty}\frac{c_{n}}{c'_{n}}\leq
  \lim_{n\to\infty}\frac{c_{n}}{d_{\wh n}}=
  \lim_{n\to\infty}\frac{c_{n}}{c_{\wh n}}\sqrt{r_{\wh n}}\leq
  \lim_{n\to\infty}\sqrt{r_{\wh n}}=0
$$
by (d). We proved (i).

(ii) By \cite{Ash}, there is $d\notin\el1$ such that $\frac{d_{n}}{c_{n}}\to 0$.
Let $c''_{n}=\max_{i\geq n}d_{i}$ and $c'_{n}=(1+2^{-n})c_{n}''$.
Then $c''$ is obviously nonincreasing and therefore $c'$ is strictly decreasing
and since $c'\geq c$, we also have $c'\notin\el1$.
To prove that $\frac{c_{n}}{c'_{n}}\to\infty$, let $\eps>0$. Since $\frac{d_{n}}{c_{n}}\to 0$,
there is $N\in\nset$ such that $d_{i}<\eps c_{i}$ for all $i\geq N$. So for every $n\geq N$
$c''_{n}\leq\eps c_{n}$ and thus $\frac{c'_{n}}{c_{n}}\leq(1+2^{-n})\eps$ and
$\frac{c'_{n}}{c_{n}}\to 0$ follows, which is enough.
\end{proof}

\begin{lem}\label{ash2}
Let $\phi$ be a gauge and $s\in\SEQ$.
\begin{enum}
\item If $\phi\circ s\in\el1$,
then there is an increasing gauge $\psi\prec\phi$ such that $\psi\circ s\in\el1$.
\item If $\phi\circ s\notin\el1$,
then there is an increasing gauge $\psi\succ\phi$ such that $\psi\circ s\notin\el1$.
\end{enum}
\end{lem}
\begin{proof}
(i) Let $c=\phi\circ s$ and let $c'$ be the sequence from the above Lemma~\ref{ash}(i).
Let $\psi$ be a piecewise linear function with breakpoints $\psi(s_{n+1})=c'_n$, $n\in\Nset$.
Since $c'$ is strictly decreasing, $\psi$ is an increasing gauge.
Since $c'\in\el1$, we have $\psi\circ s\in\el1$.
Since $\frac{c'_{n}}{c_{n}}\to\infty$, we have
\[
  \liminf_{r\to0}\frac{\psi(r)}{\phi(r)}
  \geq\liminf_{n\to\infty}\frac{\psi(s_{n+1})}{\phi(s_{n+1})}
  \geq\liminf_{n\to\infty}\frac{c'_n}{\phi(s_n)}
  =\liminf_{n\to\infty}\frac{c'_n}{c_n}=\infty,
\]
so $\psi\prec\phi$.

(ii) Let $c=\phi\circ s$ and let $c'$ be the sequence from the above Lemma~\ref{ash}(ii).
For each $n$ choose $a_n$ such that
$\max(\frac{c'_n}{2},c'_{n+1})\leq a_n<c'_n$ and define
$\psi$ as follows: the graph of $\psi$ on the interval $[s_{n},s_{n-1})$ is a segment with
endpoints $(s_n,a_n)$ and $(s_{n-1},c'_n)$.
Since $a_n<c_n'$, $\psi$ is an increasing gauge.
Since $c'\notin\el1$ and $a_n\geq \frac{c'_n}{2}$, we have $\psi\circ s\notin\el1$.
Since $\frac{c_{n}}{c'_{n}}\to\infty$, we have
\[
  \lim_{r\to0}\frac{\psi(r)}{\phi(r)}\leq\limsup_{n\to\infty}\frac{\psi(s_{n-1})}{\phi(s_n)}
  =\limsup_{n\to\infty}\frac{c'_n}{c_n}=0,
\]
so $\psi\succ\phi$.
\end{proof}

\begin{lem}\label{doubling2}
If $\phi$ is a $2^\alpha$-doubling gauge, then $\phi\prece\alpha$.
\end{lem}
\begin{proof}
Let $r>0$. There is $n\in\Nset$ such that $1\leq 2^n r\leq 2$. Since 
$\phi$ is $2^\alpha$-doubling, recursion yields
$\phi(2^nr)\leq2^{\alpha n}\phi(r)$.
Therefore,
\[
  \frac{r^\alpha}{\phi(r)}\leq\frac{r^\alpha}{\phi(2^nr)/2^{n\alpha}}\leq 
  \frac{(2^nr)^\alpha}{\phi(1)}\leq\frac{2^\alpha}{\phi(1)}
\]
and $\phi\prece\alpha$ follows.
\end{proof}
\begin{lem}\label{doubling11}
For every $L$-doubling gauge $\phi$ there is a continuous, strictly increasing,
strictly $L$-doubling gauge $\psi\approx\phi$.
\end{lem}
\begin{proof}
We first prove the lemma for $L=2$.
Let $r_n=\inf\{r:\phi(r)\geq2^{-n}\}$. Since $\phi$ is $2$-doubling, we have
$r_{n+1}\leq \frac12 r_n$. Indeed, if $r_{n+1}> \frac12 r_n$, then $\phi(\frac12 r_n)<\frac1{2^{n+1}}$, so $(\phi(r_n)\leq 2\phi(\frac12r_n)<\frac1{2^n}$, a contradiction. For $r\in[r_n,r_{n-1})$ define $\phi_1(r)=2^{-n}(1-2^{-n}) = \frac{2^n-1}{4^n}$.
Routine calculation proves that $\phi_1$ is strictly $2$-doubling and $\phi_1\approx\phi$.

Let $\phi_2(r)=\phi_1(r)+r$. It is easy to check that $\phi_2$
is strictly $2$-doubling and it is clearly increasing.
Lemma~\ref{doubling2} yields $\phi_2\prece 1$ and consequently $\phi_2\approx\phi_1$.

Define $\psi$ as follows: Let $a_n=\phi_2(2^{-n})$. Note that $a_n<2a_{n+1}$
since $\phi_2$ is strictly $2$-doubling.
Let $\psi(2^{-n})=a_n$ for all $n$ and let the graph of $\psi$
on the interval $[2^{-n-1},2^{-n}]$ be the segment with endpoints $(2^{-n-1},a_{n+1})$
and $(2^{-n},a_{n})$. In particular, $\psi$ is strictly increasing and continuous.

For $2^{-n-1}\leq r\leq 2^{-n}$
\begin{align*}
  \tfrac12\phi_2(r)&\leq\tfrac12\phi_2(2^{-n})
  =\tfrac12 a_n\leq a_{n+1}=\psi(2^{-n-1})\leq\psi(r)\\
  &\leq\psi(2^{-n})=a_n\leq 2a_{n+1}\leq 2\phi_2(2^{-n-1})\leq 2\phi_2(r)
\end{align*}
so $\psi\approx\phi_2\approx\phi_1\approx\phi$.

It remains to show that $\psi$ is strictly $2$-doubling. The formula for $\psi$ in
$[2^{-n-1},2^{-n}]$ is
\begin{align*}
  \psi(r)&=2^{n+1}(a_n-a_{n+1})(r-\frac1{2^n})+a_{n}.
\end{align*}
Thus,
\begin{align*}
  \psi(2r)&=2^{n}(a_{n-1}-a_{n})(2r-\frac1{2^{n-1}})+a_{n-1}.
\end{align*}
Since $a_n<2a_{n+1}$ and $a_{n-1}<2a_n$, we have
\begin{align*}
  \psi&(2r)-2\psi(r)\\
  &=2^n (a_{n-1}-a_n)(2r-\frac1{2^{n-1}})+a_{n-1}-2^{n+2}(a_{n}-a_{n+1})(r-\frac1{2^{n}})-2a_n\\ 
  &=   a_{n-1}(2^{n+1}r-1)+a_n(4-3\cdot2^{n+1}r)+a_{n+1}(2\cdot2^{n+1}r-4)\\
  & = a_{n-1}(2^{n+1}r-1) -2a_n(2^{n+1}r-1) +a_n(2-2^{n+1}r)-2a_{n+1}(2-2^{n+1}r)\\
  &=(a_{n-1}-2a_n)(2^{n+1}r-1)+(a_n-2a_{n+1})(2-2^{n+1}r)<0.
\end{align*}
The proof for $L=2$ is complete. For $L\neq2$ let $\beta=\log 2/\log L$ and
consider the gauge $\phi'=\phi^\beta$. Then $\phi'$ is $2$-doubling,
so by the above there is $\psi'\approx\phi'$ continuous, strictly increasing and
strictly $2$-doubling.
Let $\psi=(\psi')^{1/\beta}$.
\end{proof}

\begin{lem}\label{2gauges}
Suppose that $X$ is analytic, $\zeta\prec\xi$ are gauges and any of the following
conditions is met: $X$ is ultrametric, $X$ is doubling, $\xi$ is doubling, $X=\el2$.
If $\hm^\xi(X)>0$, then $\NN(\hm^\zeta)\neq\NN(\hm^\xi)$.
\end{lem}
\begin{proof}
This is a trivial consequence of Howroyd Theorem~\ref{howroyd}:
if $X$ is ultrametric or doubling or $\xi$ is doubling, then there is (a compact set)
$E\subs X$ such that $0<\hm^\xi(E)<\infty$ and since $\zeta\prec\xi$, we have
$\hm^\zeta(E)=0$, so $E$ witnesses $\NN(\hm^\zeta)\neq\NN(\hm^\xi)$.
It remains to handle the case $X=\el2$.
Choose $r\in\SEQ$ such that $\liminf_{n\to\infty}2^n\xi(r_n)=0$ while
$\liminf_{n\to\infty}2^n\zeta(r_n)>0$. By Lemma~\ref{cubes1}, the Cantor cube
$\Cube(2,r)$ witnesses $\NN(\hm_{\el2}^\zeta)\neq\NN(\hm_{\el2}^\xi)$, because
it isometrically embeds into $\el2$.
\end{proof}

\begin{lem}\label{2gauges2}
Suppose that $\phi\circ s\in\el1$. Then there are gauges $\zeta\prec\xi$
such that
\begin{enum}
\item $\zeta\prec\xi\prec\phi^\alpha$ for all $\alpha>1$,
\item $\zeta\circ s\in\el1$,
\item if $\phi$ is $L$-doubling and $Q>L$, then $\zeta,\xi$ are $Q$-doubling.
\end{enum}
\begin{proof}
We may suppose $\phi$ is continuous, strictly increasing and for all $n$, $\phi(s_n) \leq 1$.
For $\beta>1$ let
\[
  \tau_\beta(t)=\frac{t}{(\log\frac1t)^\beta}, \qquad \psi_\beta=\tau_\beta\circ\phi.
\]
and define $\zeta=\psi_3$, $\xi=\psi_2$.

For each $\alpha>1$
\[
  \lim_{r\to0}\frac{\phi^\alpha(r)}{\psi_\beta(r)}=\lim_{t\to0}\frac{t^\alpha}{\tau_\beta(t)}
  =\lim_{t\to0}t^{\alpha-1}(\log\tfrac1t)^\beta=0,
\]
hence $\psi_\beta\prec\phi^\alpha$ and (i) follows.
On the other hand, since
\[
  \sum_{n\in\Nset}\psi_\beta(s_n)=\sum_{n\in\Nset}\tau_\beta(\phi(s_n))=
  \sum_{n\in\Nset}\frac{\phi(s_n)}{(\log \frac1{\phi(s_n)})^\beta}\leq \sum_{n\in\Nset}\phi(s_n)
  <\infty,
\]
i.e., $\psi_\beta\circ s\in\el1$ and (ii) follows.

For any $\beta>1$ and for all $r>0$ small enough
\[
  \frac{\psi_\beta(2r)}{\psi_\beta(r)}=
  \frac{\phi(2r)}{\phi(r)}\Bigl(\frac{\log\phi(r)}{\log\phi(2r)}\Bigr)^\beta
  \leq L\Bigl(\frac{\log\phi(r)}{\log\phi(r)+1}\Bigr)^\beta<Q
\]
since the term $\bigl(\frac{\log\phi(r)}{\log\phi(r)+1}\bigr)^\beta$ is close to $1$ for
$r$ small enough. We proved (iii).
\end{proof}
\end{lem}

\begin{lem}\label{ideals}
For each gauge $\phi$ and each $2$-doubling gauge $\psi$
\[
\phi\prece\psi\Longleftrightarrow\NN(\hm^\phi_\Rset)\subs\NN(\hm^\psi_\Rset).
\]
\end{lem}
\begin{proof}
The forward implication is trivial.

By Lemma~\ref{doubling11}, we may suppose that $\psi$ is strictly $2$-doubling, continuous and
strictly increasing. To prove the reverse implication, suppose that
$\phi\not\prece\psi$ and let $r_n=\psi^{-1}(2^{-n})$. Since $\psi$ is strictly $2$-doubling,
we have $r_{n+1}<\frac12 r_n$ for all $n$.
Let $C=\cube(r)\subs[0,1]$ be the symmetric Cantor set as defined in Section~\ref{sec:haus}.
Lemma~\ref{cubes2}(ii) yields $0<\hm^\psi(C)$.
Now fix $\eps, \del>0$. Since $\phi\not\prece\psi$,
there is $r<\del$ such that $\phi(r)<\eps\psi(r)$.
Let $r_n\leq r<r_{n-1}$. Then
\[
  \phi(r_n)\leq\phi(r)<\eps\psi(r)\leq\eps\,2^{-(n-1)}.
\]
Let $\VV=\{I_p:p\in2^n\}$. It is obviously a cover of $C$ by sets of diameter $r_n$.
Therefore,
\[
  \hm^\phi_\del(C)\leq\sum_{V\in\VV}\phi(\diam V)
  =\sum_{p\in 2^n}\phi(\diam I_p)
  =2^n\phi(r_n)<2^n\eps\,2^{-(n-1)}
  =2\eps.
\]
Since $\eps$ and $\del$ were arbitrary, it follows that $\hm^\phi(C)=0$.
Overall, $C\in\NN(\hm^\phi)\setminus\NN(\hm^\psi)$ witnessing
$\NN(\hm^\phi)\nsubseteq\NN(\hm^\psi)$.
\end{proof}

\begin{lem}\label{dim2}
Let $\alpha\in[0,1)$. Let $E\subs\Rset$, $\hdim E\leq\alpha$.
Then there is a continuous gauge $\phi$ such that $\hm^\phi(E)=0$ and
$\phi\prec\beta$ for all $\beta>\alpha$.
\end{lem}
\begin{proof}
For each $n$ let $\phi_n(r)=r^{\alpha+1/n}$. Then clearly $\phi_{n+1}\prec\phi_n$
and hence, by~\cite{MR160860}, there is a gauge $\psi$ such that $\hm^\psi(E)=0$ and
$\psi\prec\phi_n$ for all $n$.
We just have to modify $\psi$ to make it continuous.
By the latter property, there is a sequence $r_n\searrow0$ such that
$\psi\rest[r_{n+1},r_n]\geq\phi_n$.
Let $s_0=r_0$, $s_1=r_1$ and if $r_n$ has been defined, let
\begin{itemyze}
\item $s'_n=\phi_{n+1}^{-1}(\phi_n(s_n))$,
\item $s_{n+1}<\min(r_{n+1},s'_n)$.
\end{itemyze}
Let
\[
  \phi(r)=
  \begin{cases}
    \phi_n(s_n)   &\text{for $r\in[s'_n,s_n]$,}\\
    \phi_{n+1}(r) &\text{for $r\in[s_{n+1},s'_n]$.}
  \end{cases}
\]
Then $\phi\prece\phi_n$ for all $n$ and $\phi\leq\psi$,
hence $\hm^\phi(E)\leq\hm^\psi(E)=0$, and $\phi$ is clearly continuous.
\end{proof}

\section{Remarks and questions}\label{sec:rem}

\subsection*{Dominated compact sets are typical}
D'Aniello and Maiuriello~\cite{MR4122477} proved that compact microscopic sets
are \emph{typical} among compact subsets of a given compact metric space. In more detail,
if $X$ is a compact metric space, let $\mc K(X)$ be the family of all nonempty compact
subsets of $X$. Provide $\mc K(X)$ with the \emph{Hausdorff metric}.
defined by
Then $(\mc K(X),\hausm)$ is a compact metric space.
Recall that a set in a metric space is \emph{meager} if it is
a countable union of nowhere dense sets and \emph{comeager} if it is a complement of
a meager set. A property of $K\in\mc K(X)$ is \emph{typical} if the compact subsets of $X$
that posses the property form a comeager set in $\mc K(X)$.
D'Aniello and Maiuriello proved the following.
\begin{thm}[{\cite[Proposition 2.9]{MR4122477}}]
A typical compact set $K\in\mc K([0,1]^n)$ is microscopic.
\end{thm}
Using our results and the well-known theorems on typicality of Hausdorff measure zero
their theorem can be very easily reproved and extended. We will prove
it in the setting of the so called \emph{Gromov-Hausdorff space} $\mc K_\mathsf{GH}$,
the space of all nonempty compact spaces provided with the \emph{Gromov-Hausdorff metric}
$d_\mathsf{GH}$. We refer to \cite{MR3831258} for the definitions and details.
\begin{thm}
If $S\subs\SEQ$ is countably determined,
then a typical compact set $K\in\mc K_\mathsf{GH}$ is $S$-dominated.
\end{thm}
\begin{proof}
Fix $s\in S$ and choose a gauge $\phi\approx s$.
By Corollary~\ref{sharp7}, $\NN(\hm^\phi)\subs\DD(\closs)$. By~\cite[Theorem 2.2]{MR3831258},
a typical compact set $K\in\mc K_\mathsf{GH}$ has Hausdorff measure $\hm^\phi$ zero, hence
it is $\closs$-dominated. Since $S$ is countably determined, being $S$-dominated is
a countable intersection of typical properties and thus is typical.
\end{proof}

\subsection*{Typical continuous functions}
A property of a continuous function on $[0,1]$ is \emph{typical} if the family of all
continuous functions possessing the property is comeager in $C([0,1])$.
Karasi\'nska, A. and Wagner-Bojakowska prove in~\cite{MR2429522} that
a typical function is one-to-one except a microscopic set.
We point out that their proof actually shows that if $\mc J$ is any $G_\del$-based \si ideal
including all singletons, then a typical function is one-to-one except
a set $J\in\mc J$. In particular, it holds for $\mc J=\DD(S)$ for any
countably determined set $S\subs\SEQ$.

\subsection*{Microscopic symmetric Cantor sets}
Balcerzak, Filipczak and Nowakowski~\cite{MR4036593} investigate the family $\CS$ of all
symmetric Cantor sets $\cube(r)$ (cf.~Section~\ref{sec:haus}).
They consider, for each $f\colon \Nset\to[1,\infty)$, the family $\MM(f)\subs\CS$ defined thus:
\[
  \MM(f)=\{\cube(r):\forall\eps>0\ \emany n\ r_n<\eps^{f(n)}\}
\]
and compare microscopic symmetric Cantor sets with these families for $f$
taking a particular form $f(n)=p^n$, where $p\in(1,\infty)$.
They prove (recall that microscopic sets constitute the ideal $\DD(\clos\geom)$):
\begin{thm}[{\cite[3.2,3.6]{MR4036593}}]\label{piotr1}
$\MM(2^n)\subsetneqq\DD(\clos\geom)\cap\CS\subs\bigcap_{1<p<2}\MM(p^n)$
\end{thm}
\noindent
and they ask (\cite[3.7]{MR4036593}) if the latter inclusion is proper.
In order to answer their question we first characterize $\MM(p^n)$.
Let $\micr(r)=-\frac{1}{\log r}$ (cf.~Example~\ref{ex4.2}(i)).
\begin{lem}\label{CS1}
If $p\geq2$, then
$C(r)\in\MM(p^n)$ if and only if $\hm^{\micr^{1\lom\log p}}(C(r))=0$.
\end{lem}
\begin{proof}
It is an immediate consequence of classical density theorems
(e.g., \cite[Theorem 4.15]{MR2288081}) that if $\phi$ is a doubling gauge, then
$\hm^\phi(C(r))=0$ if and only if $\liminf_{n\to\infty}2^n\phi(r_n)=0$.
Since $\micr^{1\lom\log p}$ is doubling for each $p\geq2$, we have
\begin{align*}
  \hm^{\micr^{1\lom\log p}}(C(r))=0&\Longleftrightarrow
  \liminf_{n\to\infty}2^n\micr^{1\lom\log p}(r_n)=0
  \Longleftrightarrow
  \liminf_{n\to\infty}2^n\bigl(-\tfrac1{\log r_n}\bigr)^{1\lom\log p}=0 \\
  &\Longleftrightarrow
  \liminf_{n\to\infty}\bigl(-\tfrac{p^n}{\log r_n}\bigr)^{1\lom\log p}=0
  \Longleftrightarrow
  \liminf_{n\to\infty}\bigl(-\tfrac{p^n}{\log r_n}\bigr)=0
\end{align*}
and the latter is by~\cite[Proposition 2.1]{MR4036593} equivalent with $C(r)\in\MM(p^n)$.
\end{proof}
\begin{prop}
$\MM(2^n)\subsetneqq\DD(\clos\geom)\cap\CS\subsetneqq\bigcap_{1<p<2}\MM(p^n)$.
\end{prop}
\begin{proof}
By Lemma~\ref{CS1}, the inclusions read
\begin{equation}\label{piotreK}
\NN(\hm^{\micr})\cap\CS\subsetneqq\DD(\clos\geom)\cap\CS
  \subsetneqq\bigcap_{\alpha>1}\NN(\hm^{\micr^\alpha})\cap\CS.
\end{equation}
Since $\micr\asymp\geom$ and it is $L$-doubling for each $L>1$,
both follow from Theorem~\ref{sharp8}(ii). We only have to remind that
the sets witnessing both
$\NN(\hm^{\micr})\neq\DD(\clos\geom)$ and
$\DD(\clos\geom)\neq\bigcap_{\alpha>1}\NN(\hm^{\micr^\alpha})$
are symmetric Cantor sets so that they witness also \eqref{piotreK}.
\end{proof}

%
%

\subsection*{Dominated sets and Hausdorff measures}
We know from Theorem~\ref{versus1} that the inclusion $\DD(\closs)\subs\NN(\hm^\phi)$ is
characterized by $\phi\circ s\in\el1$.
As to the other inclusion $\NN(\hm^\phi)\subs\DD(\closs)$,
we only know from Corollary~\ref{sharp7} that it is guaranteed by $\phi\asymp s$.
\begin{question}
Is there a characterization of $\NN(\hm^\phi)\subs\DD(\closs)$
in terms of $\phi$ and $s$?
\end{question}

We just mentioned that Corollary~\ref{sharp7} claims that
if $\phi\asymp s$ then every $\hm^\phi$-null set is $\closs$-dominated.
What about sets of \si finite measure?
\begin{question}
Suppose $\phi\asymp s$. Is then $\NNs(\hm^\phi)\subs\DD(\closs)$?
\end{question}

\subsection*{Additively complete sets}
Generally we do not know much about the following problems regarding the relationship
of the ideals $\DD(\closs)$ and $\DD(\whs)$.
Section~\ref{sec:shift} contains a few partial answers to the questions listed therein.
Lot of questions remain unanswered.
%
%
\begin{questions}
(i) Let $s\in\SEQ$. Is there a separable (or compact) metric space
$X$ that is $\whs$-dominated but not $\closs$-dominated?

(ii) Let $s\in\SEQ$ be $2$-doubling. Is there a (compact) set $E\subs\Rset$ that is
$\whs$-dominated but not $\closs$-dominated?
\end{questions}

\begin{questions}
Some specific questions:
\begin{enum}
\item \emph{Microscopic sets:} Is the family $\DD(\wh\geom)$ an ideal or even a \si ideal?
What about $\DD(\wh{\geom_p})$ for $p<2$?
\item \emph{Generalized harmonic sequences:} Is $\DD(\wh{\harm^p})=\DD(\clos{\harm^p})$?
\end{enum}
\end{questions}

A key fact in the proof of Theorem~\ref{thmwarpj} is that in a doubling space
$\DDs(\del\whs)=\DDs(\whs)$.
\begin{question}
Suppose $X$ is doubling, $\del>0$. Is $\DD(\del\whs)=\DD(\whs)$?
\end{question}

\subsection*{Gauges}
This question is related to Lemma~\ref{ash2}.
\begin{question}
Suppose that $\phi$ is doubling and $\phi\circ s\in\el1$. Is there a doubling
$\psi\prec\phi$ such that $\psi\circ s\in\el1$?
\end{question}

\subsection*{Inscribing Theorem?}
Many \si ideals have the following property: Every $G_\del$-set in a compact space
that is not small contains a compact set that is not small.
Do the ideals of dominated sets have this property? Or, to begin with, just the microscopic sets on the line? This question seems
to be rather difficult and important.
\begin{questions}
Let $E$ be a $G_\del$-set in a compact space and $s\in\SEQ$. Are the following true?
\begin{enum}
\item If $E$ is not $\closs$-dominated, then it contains a compact set that is
not $\closs$-dominated.
\item If $E$ is not microscopic, then it contains a compact set that is
not microscopic.
\end{enum}
\end{questions}

\medskip
We would like to thank to Arturo Mart\'\i nez-Celis, Jaros\l aw Swaczyna and Cristina Villanueva-Segovia for inspiring discussions.

\bibliographystyle{amsplain}
\bibliography{micro_list,micro_extra}

@Book{MR0281862,
  author     = {Rogers, C. A.},
  title      = {Hausdorff measures},
  pages      = {viii+179},
  publisher  = {Cambridge University Press, London-New York},
  address    = {London},
  mrclass    = {28.13},
  mrnumber   = {0281862 (43 \#7576)},
  mrreviewer = {Edwin Hewitt},
  year       = {1970},
}

@article{Ash,
  author = {J. M. Ash},
  title = {Neither a Worst Convergent Series nor a Best Divergent Series Exists},
  journal = {The College Mathematics Journal},
  volume = {28},
  number = {4},
  pages = {296--297},
  year = {1997},
  publisher = {Taylor \& Francis},
  doi = {10.1080/07468342.1997.11973879},
  URL = {https://doi.org/10.1080/07468342.1997.11973879},
  eprint = {https://doi.org/10.1080/07468342.1997.11973879}
}

@article {MR695109,
    AUTHOR = {Timan, A. F. and Vestfrid, I. A.},
     TITLE = {Any separable ultrametric space is isometrically embeddable in
              {$l\sb{2}$}},
   JOURNAL = {Funktsional. Anal. i Prilozhen.},
  FJOURNAL = {Akademiya Nauk SSSR. Funktsional\cprime ny\u i\ Analiz i ego
              Prilozheniya},
    VOLUME = {17},
      YEAR = {1983},
    NUMBER = {1},
     PAGES = {85--86},
      ISSN = {0374-1990},
   MRCLASS = {54E35 (46P05 54C25)},
  MRNUMBER = {695109},
}

@article {MR160860,
    AUTHOR = {Rogers, C. A. and Taylor, S. J.},
     TITLE = {Additive set functions in {E}uclidean space. {II}},
   JOURNAL = {Acta Math.},
  FJOURNAL = {Acta Mathematica},
    VOLUME = {109},
      YEAR = {1963},
     PAGES = {207--240},
      ISSN = {0001-5962,1871-2509},
   MRCLASS = {28.13},
  MRNUMBER = {160860},
MRREVIEWER = {Chr.\ Y.\ Pauc},
       DOI = {10.1007/BF02391813},
       URL = {https://doi.org/10.1007/BF02391813},
}

@Book{MR867284,
  author     = {Falconer, K. J.},
  title      = {The geometry of fractal sets},
  isbn       = {0-521-25694-1; 0-521-33705-4},
  pages      = {xiv+162},
  publisher  = {Cambridge University Press},
  series     = {Cambridge Tracts in Mathematics},
  volume     = {85},
  address    = {Cambridge},
  mrclass    = {28-02 (28A99 60J65)},
  mrnumber   = {867284 (88d:28001)},
  mrreviewer = {K. E. Hirst},
  year       = {1986},
}

@InCollection{MR4146582,
  author    = {Hru\v{s}\'{a}k, M. and Zindulka, O.},
  booktitle = {Centenary of the {B}orel conjecture},
  title     = {Strong measure zero in {P}olish groups},
  doi       = {10.1090/conm/755/15184},
  pages     = {37--68},
  publisher = {Amer. Math. Soc., [Providence], RI},
  series    = {Contemp. Math.},
  url       = {https://doi.org/10.1090/conm/755/15184},
  volume    = {755},
  mrclass   = {03E17 (22A10 22B05 54E52)},
  mrnumber  = {4146582},
  year      = {[2020] \copyright 2020},
}

@Article{MR1317515,
  author     = {Howroyd, J. D.},
  title      = {On dimension and on the existence of sets of finite positive {H}ausdorff measure},
  doi        = {10.1112/plms/s3-70.3.581},
  issn       = {0024-6115},
  number     = {3},
  pages      = {581--604},
  url        = {http://dx.doi.org/10.1112/plms/s3-70.3.581},
  volume     = {70},
  coden      = {PLMTAL},
  fjournal   = {Proceedings of the London Mathematical Society. Third Series},
  journal    = {Proc. London Math. Soc. (3)},
  mrclass    = {28A78},
  mrnumber   = {1317515 (96b:28004)},
  mrreviewer = {Hermann Haase},
  year       = {1995},
}

@Book{MR1333890,
  author     = {Mattila, P.},
  title      = {Geometry of sets and measures in {E}uclidean spaces},
  doi        = {10.1017/CBO9780511623813},
  isbn       = {0-521-46576-1; 0-521-65595-1},
  note       = {Fractals and rectifiability},
  pages      = {xii+343},
  publisher  = {Cambridge University Press},
  series     = {Cambridge Studies in Advanced Mathematics},
  url        = {http://dx.doi.org/10.1017/CBO9780511623813},
  volume     = {44},
  address    = {Cambridge},
  mrclass    = {28A75 (49Q20)},
  mrnumber   = {1333890 (96h:28006)},
  mrreviewer = {Harold Parks},
  year       = {1995},
}

@article {MR3831258,
    AUTHOR = {Jurina, S. and MacGregor, N. and Mitchell, A. and Olsen, L.
              and Stylianou, A.},
     TITLE = {On the {H}ausdorff and packing measures of typical compact
              metric spaces},
   JOURNAL = {Aequationes Math.},
  FJOURNAL = {Aequationes Mathematicae},
    VOLUME = {92},
      YEAR = {2018},
    NUMBER = {4},
     PAGES = {709--735},
      ISSN = {0001-9054,1420-8903},
   MRCLASS = {28A78 (28A80)},
  MRNUMBER = {3831258},
MRREVIEWER = {J\"org\ Neunh\"auserer},
       DOI = {10.1007/s00010-018-0548-5},
       URL = {https://doi.org/10.1007/s00010-018-0548-5},
}

@article {MR2288081,
    AUTHOR = {Edgar, G. A.},
     TITLE = {Centered densities and fractal measures},
   JOURNAL = {New York J. Math.},
  FJOURNAL = {New York Journal of Mathematics},
    VOLUME = {13},
      YEAR = {2007},
     PAGES = {33--87},
      ISSN = {1076-9803},
   MRCLASS = {28A78 (28A12 28A80)},
  MRNUMBER = {2288081},
MRREVIEWER = {Lars\ Olsen},
       URL = {http://nyjm.albany.edu:8000/j/2007/13_33.html},
}

@unpublished{invariants2,
    author = {M. Hru\v{s}\'{a}k and Zindulka, O. },
    title = {Fractal measures and special subsets of reals},
    note = {to appear}
}

@article {MR4185783,
    AUTHOR = {Karasi\'nska, A.},
     TITLE = {Remarks on some generalization of the notion of microscopic
              sets},
   JOURNAL = {Math. Slovaca},
  FJOURNAL = {Mathematica Slovaca},
    VOLUME = {70},
      YEAR = {2020},
    NUMBER = {6},
     PAGES = {1349--1356},
      ISSN = {0139-9918,1337-2211},
   MRCLASS = {28A05 (28A75)},
  MRNUMBER = {4185783},
MRREVIEWER = {Riccardo\ Cristoferi},
       DOI = {10.1515/ms-2017-0436},
       URL = {https://doi.org/10.1515/ms-2017-0436},
}

@article {MR4426649,
    AUTHOR = {Agadzhanov, A. N.},
     TITLE = {Peano-type curves, {L}iouville numbers, and microscopic sets},
      NOTE = {Translated from Dokl. Akad. Nauk {\bf 485} (2019), No. 1.},
   JOURNAL = {Dokl. Math.},
  FJOURNAL = {Doklady Mathematics},
    VOLUME = {99},
      YEAR = {2019},
    NUMBER = {2},
     PAGES = {117--120},
      ISSN = {1064-5624,1531-8362},
   MRCLASS = {54C99 (11K55 54F15)},
  MRNUMBER = {4426649},
       DOI = {10.1134/s1064562419020017},
       URL = {https://doi.org/10.1134/s1064562419020017},
}

@article {MR3869249,
    AUTHOR = {Paszkiewicz, A.},
     TITLE = {On microscopic sets and {F}ubini property in all directions},
   JOURNAL = {Math. Slovaca},
  FJOURNAL = {Mathematica Slovaca},
    VOLUME = {68},
      YEAR = {2018},
    NUMBER = {5},
     PAGES = {1041--1048},
      ISSN = {0139-9918,1337-2211},
   MRCLASS = {28A05 (03E15 26A30)},
  MRNUMBER = {3869249},
       DOI = {10.1515/ms-2017-0165},
       URL = {https://doi.org/10.1515/ms-2017-0165},
}

@article {MR3759529,
    AUTHOR = {Horbaczewska, G.},
     TITLE = {General approach to microscopic-type sets},
   JOURNAL = {J. Math. Anal. Appl.},
  FJOURNAL = {Journal of Mathematical Analysis and Applications},
    VOLUME = {461},
      YEAR = {2018},
    NUMBER = {1},
     PAGES = {51--58},
      ISSN = {0022-247X,1096-0813},
   MRCLASS = {28A05 (28A75 28A80)},
  MRNUMBER = {3759529},
       DOI = {10.1016/j.jmaa.2018.01.007},
       URL = {https://doi.org/10.1016/j.jmaa.2018.01.007},
}

@article {MR3823475,
    AUTHOR = {Chrz\c{a}szcz, K. and G\l\c{a}b, S.},
     TITLE = {Isomorphism theorems for {$\sigma$}-ideals of microscopic sets
              in various metric spaces},
   JOURNAL = {Bull. Soc. Sci. Lett. \L\'od\'z{} S\'er. Rech. D\'eform.},
  FJOURNAL = {Bulletin de la Soci\'et\'e{} des Sciences et des Lettres de
              \L\'od\'z. S\'erie: Recherches sur les D\'eformations},
    VOLUME = {67},
      YEAR = {2017},
    NUMBER = {3},
     PAGES = {127--139},
      ISSN = {0459-6854,2450-9329},
   MRCLASS = {54H05},
  MRNUMBER = {3823475},
}

@article {MR3685162,
    AUTHOR = {Karasi\'nska, A. and Paszkiewicz, A. and
              Wagner-Bojakowska, E.},
     TITLE = {A generalization of the notion of microscopic sets},
   JOURNAL = {Lith. Math. J.},
  FJOURNAL = {Lithuanian Mathematical Journal},
    VOLUME = {57},
      YEAR = {2017},
    NUMBER = {3},
     PAGES = {319--330},
      ISSN = {0363-1672,1573-8825},
   MRCLASS = {28A05 (03E15 28A75 54E52 54H05)},
  MRNUMBER = {3685162},
MRREVIEWER = {Szymon\ \.Zeberski},
       DOI = {10.1007/s10986-017-9363-2},
       URL = {https://doi.org/10.1007/s10986-017-9363-2},
}

@article {MR3568089,
    AUTHOR = {Czudek, K. and Kwela, A. and Mro\.zek, N. and Wo{\l}oszyn, W.},
     TITLE = {Ideal-like properties of generalized microscopic sets},
   JOURNAL = {Acta Math. Hungar.},
  FJOURNAL = {Acta Mathematica Hungarica},
    VOLUME = {150},
      YEAR = {2016},
    NUMBER = {2},
     PAGES = {269--285},
      ISSN = {0236-5294,1588-2632},
   MRCLASS = {28A05 (03E15 26A30)},
  MRNUMBER = {3568089},
       DOI = {10.1007/s10474-016-0659-1},
       URL = {https://doi.org/10.1007/s10474-016-0659-1},
}

@article {MR3529313,
    AUTHOR = {Paszkiewicz, A. and Wagner-Bojakowska, E.},
     TITLE = {Fubini property for microscopic sets},
   JOURNAL = {Tatra Mt. Math. Publ.},
  FJOURNAL = {Tatra Mountains Mathematical Publications},
    VOLUME = {65},
      YEAR = {2016},
     PAGES = {143--149},
      ISSN = {1210-3195,1338-9750},
   MRCLASS = {28A05 (28A75)},
  MRNUMBER = {3529313},
       DOI = {10.1515/tmmp-2016-0012},
       URL = {https://doi.org/10.1515/tmmp-2016-0012},
}

@article {MR3482702,
    AUTHOR = {Kwela, A.},
     TITLE = {Additivity of the ideal of microscopic sets},
   JOURNAL = {Topology Appl.},
  FJOURNAL = {Topology and its Applications},
    VOLUME = {204},
      YEAR = {2016},
     PAGES = {51--62},
      ISSN = {0166-8641,1879-3207},
   MRCLASS = {28A05},
  MRNUMBER = {3482702},
MRREVIEWER = {Benjamin\ Vejnar},
       DOI = {10.1016/j.topol.2016.01.031},
       URL = {https://doi.org/10.1016/j.topol.2016.01.031},
}

@article {MR3384702,
    AUTHOR = {Wagner-Bojakowska, E. and Wilczy\'nski, W.
              and Wojdowski, W.},
     TITLE = {Density topology involving microscopic sets and category},
   JOURNAL = {Tatra Mt. Math. Publ.},
  FJOURNAL = {Tatra Mountains Mathematical Publications},
    VOLUME = {62},
      YEAR = {2015},
     PAGES = {113--132},
      ISSN = {1210-3195,1338-9750},
   MRCLASS = {54A10 (28A05)},
  MRNUMBER = {3384702},
       DOI = {10.1515/tmmp-2015-0009},
       URL = {https://doi.org/10.1515/tmmp-2015-0009},
}

@article {MR3259051,
    AUTHOR = {Karasi\'nska, A. and Wagner-Bojakowska, E.},
     TITLE = {Microscopic and strongly microscopic sets on the plane.
              {F}ubini theorem and {F}ubini property},
   JOURNAL = {Demonstr. Math.},
  FJOURNAL = {Demonstratio Mathematica},
    VOLUME = {47},
      YEAR = {2014},
    NUMBER = {3},
     PAGES = {581--594},
      ISSN = {0420-1213,2391-4661},
   MRCLASS = {28A05 (28A75)},
  MRNUMBER = {3259051},
MRREVIEWER = {Tomasz\ Filipczak},
       DOI = {10.2478/dema-2014-0046},
       URL = {https://doi.org/10.2478/dema-2014-0046},
}

@article {MR3242549,
    AUTHOR = {G. Horbaczewska},
     TITLE = {Microscopic sets with respect to sequences of functions},
   JOURNAL = {Tatra Mt. Math. Publ.},
  FJOURNAL = {Tatra Mountains Mathematical Publications},
    VOLUME = {58},
      YEAR = {2014},
     PAGES = {137--144},
      ISSN = {1210-3195,1338-9750},
   MRCLASS = {28A05 (03E15)},
  MRNUMBER = {3242549},
MRREVIEWER = {Gareth\ Speight},
       DOI = {10.2478/tmmp-2014-0012},
       URL = {https://doi.org/10.2478/tmmp-2014-0012},
}

@incollection {MR3204595,
    AUTHOR = {Horbaczewska, G. and Karasi\'nska, A. and
              Wagner-Bojakowska, E.},
     TITLE = {Properties of the {$\sigma$}-ideal of microscopic sets},
 BOOKTITLE = {Traditional and present-day topics in real analysis},
     PAGES = {325--343},
 PUBLISHER = {Faculty of Mathematics and Computer Science. University of
              \L\'od\'z, \L\'od\'z},
      YEAR = {2013},
      ISBN = {978-83-7525-971-1},
   MRCLASS = {28A05 (54H05)},
  MRNUMBER = {3204595},
MRREVIEWER = {Marek\ Balcerzak},
}

@article {MR2904079,
    AUTHOR = {A. Karasi\'nska and Wagner-Bojakowska, E. },
     TITLE = {Homeomorphisms of linear and planar sets of the first category into microscopic sets},
   JOURNAL = {Topology Appl.},
  FJOURNAL = {Topology and its Applications},
    VOLUME = {159},
      YEAR = {2012},
    NUMBER = {7},
     PAGES = {1894--1898},
      ISSN = {0166-8641,1879-3207},
   MRCLASS = {54C10 (28A05 54E52)},
  MRNUMBER = {2904079},
MRREVIEWER = {Peter\ Elia\v s},
       DOI = {10.1016/j.topol.2011.11.055},
       URL = {https://doi.org/10.1016/j.topol.2011.11.055},
}

@incollection {MR2882548,
    AUTHOR = {Karasi\'nska, A. and Poreda, W. and
              Wagner-Bojakowska, E.},
     TITLE = {Duality principle for microscopic sets},
 BOOKTITLE = {Real functions, density topology and related topics},
     PAGES = {83--87},
 PUBLISHER = {\L\'od\'z{} Univ. Press, \L\'od\'z},
      YEAR = {2011},
      ISBN = {978-83-7525-536-2},
   MRCLASS = {28A05 (54E52)},
  MRNUMBER = {2882548},
MRREVIEWER = {K.\ P. S. Bhaskara Rao},
}

@article {MR2429522,
    AUTHOR = {Karasi\'nska, A. and E. Wagner-Bojakowska},
     TITLE = {Nowhere monotone functions and microscopic sets},
   JOURNAL = {Acta Math. Hungar.},
  FJOURNAL = {Acta Mathematica Hungarica},
    VOLUME = {120},
      YEAR = {2008},
    NUMBER = {3},
     PAGES = {235--248},
      ISSN = {0236-5294,1588-2632},
   MRCLASS = {26A30 (26A45)},
  MRNUMBER = {2429522},
MRREVIEWER = {Miroslav\ Zelen\'y},
       DOI = {10.1007/s10474-008-7093-y},
       URL = {https://doi.org/10.1007/s10474-008-7093-y},
}

@article {MR2152488,
    AUTHOR = {Appell, J.},
     TITLE = {A short story on microscopic sets},
   JOURNAL = {Atti Semin. Mat. Fis. Univ. Modena Reggio Emilia},
  FJOURNAL = {Atti del Seminario Matematico e Fisico dell' Universit\`a{} di
              Modena e Reggio Emilia},
    VOLUME = {52},
      YEAR = {2004},
    NUMBER = {2},
     PAGES = {229--233},
      ISSN = {1825-1269},
   MRCLASS = {28A05 (26A16 26A30 28A80)},
  MRNUMBER = {2152488},
}

@article {MR2290215,
    AUTHOR = {J. Appell and D'Aniello, E. and V\"ath, M.},
     TITLE = {Errata and addendum to: ``{S}ome remarks on small sets''
              [{R}icerche {M}at. {\bf 50} (2001), no. 2, 255--274;
              MR1909968]},
   JOURNAL = {Ricerche Mat.},
  FJOURNAL = {Ricerche di Matematica},
    VOLUME = {54},
      YEAR = {2005},
    NUMBER = {1},
     PAGES = {211--213},
      ISSN = {0035-5038},
   MRCLASS = {28A05 (26A15 26A16 26A30)},
  MRNUMBER = {2290215},
}

@article {MR1912017,
    AUTHOR = {J. Appell},
     TITLE = {``{S}mall'' sets and operators in functional analysis},
   JOURNAL = {Rend. Istit. Mat. Univ. Trieste},
  FJOURNAL = {Rendiconti dell'Istituto di Matematica dell'Universit\`a{} di
              Trieste. An International Journal of Mathematics},
    VOLUME = {33},
      YEAR = {2001},
    NUMBER = {1-2},
     PAGES = {127--199},
      ISSN = {0049-4704,2464-8728},
   MRCLASS = {26A21 (28A05 28A80 54E52)},
  MRNUMBER = {1912017},
MRREVIEWER = {A.\ Precupanu},
}

@article {MR1909968,
    AUTHOR = {Appell, J. and D'Aniello, E. and V\"ath, M.},
     TITLE = {Some remarks on small sets},
   JOURNAL = {Ricerche Mat.},
  FJOURNAL = {Ricerche di Matematica},
    VOLUME = {50},
      YEAR = {2001},
    NUMBER = {2},
     PAGES = {255--274},
      ISSN = {0035-5038},
   MRCLASS = {28A05 (26A15 26A16 26A30)},
  MRNUMBER = {1909968},
MRREVIEWER = {El\.zbieta\ Wagner},
}

@article {MR4122477,
    AUTHOR = {D'Aniello, E. and Maiuriello, M.},
     TITLE = {On some generic small {C}antor spaces},
   JOURNAL = {Z. Anal. Anwend.},
  FJOURNAL = {Zeitschrift f\"ur Analysis und ihre Anwendungen. Journal of
              Analysis and its Applications},
    VOLUME = {39},
      YEAR = {2020},
    NUMBER = {3},
     PAGES = {277--288},
      ISSN = {0232-2064,1661-4534},
   MRCLASS = {28A05 (28A80 54E52)},
  MRNUMBER = {4122477},
       DOI = {10.4171/zaa/1660},
       URL = {https://doi.org/10.4171/zaa/1660},
}

@article {MR4036593,
    AUTHOR = {Balcerzak, M. and Filipczak, T. and Nowakowski, P.},
     TITLE = {Families of symmetric {C}antor sets from the category and
              measure viewpoints},
   JOURNAL = {Georgian Math. J.},
  FJOURNAL = {Georgian Mathematical Journal},
    VOLUME = {26},
      YEAR = {2019},
    NUMBER = {4},
     PAGES = {545--553},
      ISSN = {1072-947X,1572-9176},
   MRCLASS = {28A80 (03E15 28A35 54E52)},
  MRNUMBER = {4036593},
       DOI = {10.1515/gmj-2019-2039},
       URL = {https://doi.org/10.1515/gmj-2019-2039},
}

@article {MR4324482,
    AUTHOR = {Nowakowski, P.},
     TITLE = {The family of central {C}antor sets with packing dimension
              zero},
   JOURNAL = {Tatra Mt. Math. Publ.},
  FJOURNAL = {Tatra Mountains Mathematical Publications},
    VOLUME = {78},
      YEAR = {2021},
     PAGES = {1--8},
      ISSN = {1210-3195,1338-9750},
   MRCLASS = {28A80 (28A35 28A78 54E52)},
  MRNUMBER = {4324482},
}

@article {MR2058531,
    AUTHOR = {Donnini, C. and Martellotti, A.},
     TITLE = {Microscopic subsets of a {B}anach space and characterizations
              of the drop property},
   JOURNAL = {Sci. Math. Jpn.},
  FJOURNAL = {Scientiae Mathematicae Japonicae},
    VOLUME = {59},
      YEAR = {2004},
    NUMBER = {3},
     PAGES = {525--530},
      ISSN = {1346-0862,1346-0447},
   MRCLASS = {46B20 (46B50)},
  MRNUMBER = {2058531},
MRREVIEWER = {John\ R.\ Giles},
}

@PHDTHESIS{pospisil,
  author =       {Posp{\'\i}{\v s}il, M.},
  title =        {Microscopic sets and drops in {B}anach spaces},
  school =       {Charles University},
  year =         {2016},
  type =         {master thesis},
  address =      {Prague},
}

@article {MR2543905,
    AUTHOR = {Filipczak, M. and Wagner-Bojakowska, E.},
     TITLE = {Remarks on small sets on the real line},
   JOURNAL = {Tatra Mt. Math. Publ.},
  FJOURNAL = {Tatra Mountains Mathematical Publications},
    VOLUME = {42},
      YEAR = {2009},
     PAGES = {73--80},
      ISSN = {1210-3195,1338-9750},
   MRCLASS = {28A05 (28A75)},
  MRNUMBER = {2543905},
MRREVIEWER = {Miroslav\ Zelen\'y},
       DOI = {10.2478/v10127-009-0007-8},
       URL = {https://doi.org/10.2478/v10127-009-0007-8},
}
\end{document}